\documentclass[reqno]{amsart}
\usepackage{amsfonts,amssymb,amscd,amsmath,verbatim,calc, times,url}
\usepackage{amsthm, psfrag,latexsym,epsfig,mdwlist,graphicx,dirtytalk, mathtools, longtable}
\usepackage{extarrows}

\usepackage[capitalise, noabbrev]{cleveref}
\usepackage{enumerate}
\usepackage[shortlabels]{enumitem}
\usepackage{xcolor}
\usepackage{enumitem}
\usepackage{tikz}
\usetikzlibrary{positioning, shapes.geometric}

\usepackage{nomencl}
\usepackage{txfonts}
\usetikzlibrary{shadings}

\usepackage[utf8]{inputenc}
\usepackage{amssymb}
\usepackage{pifont}
\usepackage{tikz}
\usepackage{floatrow}
\usepackage{nomencl}
\usepackage{txfonts}
\usetikzlibrary{shadings}
\theoremstyle{plain}
\newtheorem{theorem}{Theorem}[section]
\newtheorem{lemma}[theorem]{Lemma}
\newtheorem{proposition}[theorem]{Proposition}
\newtheorem{corollary}[theorem]{Corollary}
\theoremstyle{definition}
\newtheorem{definition}[theorem]{Definition}

\newtheorem{remark}[theorem]{Remark}

\newtheorem{example}[theorem]{Example}

\newtheorem{question}[theorem]{Question}
\newtheorem{running example}[theorem]{Running Example}

\usepackage{adjustbox}

\numberwithin{equation}{section}

\DeclareMathOperator{\lcm}{lcm}

\DeclareMathOperator{\mdeg}{mdeg}
\DeclareMathOperator{\mingens}{MinGens}

\renewcommand{\P}{\mathcal{P}}

\newcommand{\FF}{\mathbb{F}}
\newcommand{\LL}{\mathbb{L}}
\newcommand{\NN}{\mathbb{N}}

\newcommand{\im}{\mathrm{im}}
\newcommand{\lk}{\mathrm{lk}}
\newcommand{\pd}{\mathrm{pd}}
\newcommand{\reg}{\mathrm{reg}}
\newcommand{\depth}{\mathrm{depth}}

\newcommand{\qand}{\quad \mbox{and}\quad }
\newcommand{\qor}{\quad \mbox{or}\quad }
\newcommand{\qif}{\quad \mbox{if}\quad }
\newcommand{\qfor}{\quad \mbox{for}\quad }
\newcommand{\qforall}{\quad \mbox{for all}\quad }
\newcommand{\qforsome}{\quad \mbox{for some}\quad }
\newcommand{\qwhere}{\quad \mathrm{where}\quad  }
\newcommand{\st}{\ \colon \ }
\newcommand{\supp}{\mathrm{Supp}}
\newcommand{\rH}{\tilde{H}}
\newcommand{\ba}{\mathbf{a}}
\newcommand{\bb}{\mathbf{b}}
\newcommand{\bc}{\mathbf{c}}
\newcommand{\bd}{\mathbf{d}}

\newcommand{\bh}{\mathbf{h}}
\newcommand{\bg}{\mathbf{g}}
\newcommand{\bx}{\mathbf{x}}

\newcommand{\bcx}{\bx^\bc}
\newcommand{\bdx}{\bx^\bd}

\newcommand{\bgx}{\bx^\bg}
\newcommand{\bm}{\mathbf{m}}
\newcommand{\m}{\bm}
\newcommand{\bn}{\mathbf{n}}
\newcommand{\bu}{\mathbf{u}}
\newcommand{\bv}{\mathbf{v}}
\newcommand{\bw}{\mathbf{w}}
\newcommand{\ssm}{\smallsetminus}

\newcommand{\introthmname}{}
\newtheorem{introthminn}{\introthmname}

\title{Building monomial ideals with similar Betti numbers}

\author[S. Faridi]{Sara Faridi}\thanks{Faridi's research is supported by
NSERC Discovery Grant 2023-05929.}
\address[S. Faridi]
{Department of Mathematics \& Statistics, 
Dalhousie University, 
6297 Castine Way, 
PO BOX 15000, 
Halifax, NS, 
Canada B3H 4R2 
}
\email{faridi@dal.ca}

\author[P. Li]{Peilin Li}
\address[P. Li]
{Department of Mathematics,
1984 Mathematics Road,
Vancouver, BC Canada V6T 1Z2
}
\email{pl254360@math.ubc.ca}
 

\begin{document}

\begin{abstract}
Motivated by the fact that as the number of generators of an ideal grows so does the complexity of calculating relations among the
generators,  this paper identifies collections of monomial ideals
  with a growing number of generators which have predictable free resolutions. We use elementary collapses from discrete homotopy theory to construct infinitely many monomial ideals, with an arbitrary number of generators, which have similar or the same betti numbers. We show that the Cohen-Macaulay property in each unmixed (pure) component of the ideal is preserved as the ideal is expanded.
 \end{abstract}
%\tableofcontents

\maketitle

%%%%%%%%%%%%%%%%%%%%%%%%%%%%%%%%%%%%%%%%%%%%%%%%%%
\section{Introduction}
%%%%%%%%%%%%%%%%%%%%%%%%%%%%%%%%%%%%%%%%%%%%%%%%%%

A minimal free resolution of an ideal generated by a set of homogeneous polynomials encodes information about the relations between those polynomials, and the relations between those relations, and so on. The minimal number of such relations in each step is encoded in (unique) sequence of numbers known as the {\it betti numbers} of the ideal.

Most algebraic invariants of the ideal
-- such as projective dimension, Castelnuovo-Mumford regularity,
Hilbert series -- can be computed via the minimal free resolution.
This gives minimal free resolutions immense value for researchers in
commutative algebra and algebraic geometry. However, computing
resolutions remains a difficult problem, where new tools
are developed to construct free resolutions for special classes of
ideals.

The focus of the current paper is on ideals generated by monomials, or
products of variables in a polynomial ring.  This class of ideals
enjoys special combinatorial properties. In particular, free
resolutions of monomial ideals and their betti numbers can be
calculated using chain complexes of cell complexes under certain conditions. When such conditions hold, we say the complex \emph{supports} a free resolution of the monomial ideal. This method was
first introduced by Diana Taylor~\cite{T} in her thesis for simplices,
and then generalized to the larger class of cell complexes by Bayer
and Sturmfels~\cite{BS}.

A natural outcome of this approach -- modeling a free resolution using
a topological object -- is the following general question: if two
topological objects have the same \say{shape}, how different can the
free resolutions they support be?

More precisely, consider two simplicial complexes that are
homotopy equivalent. How
can we compare the free resolutions supported by those two simplicial complexes? 

The homotopy equivalence we are concerned with in this paper is
{\it simplicial collapsing}: deletion of a face of a simplicial
complex that is contained in exactly one maximal face.

Suppose $\bn$ is part of the minimal monomial generating set of the
monomial ideal $I$, and $I$ has a (minimal) free resolution supported
on the simplicial complex $\Delta$ in \cref{f:intro} on the left,
where we have labeled one of the vertices with $\bn$ (and the other
vertices will be similarly labeled with the other generators of $I$).

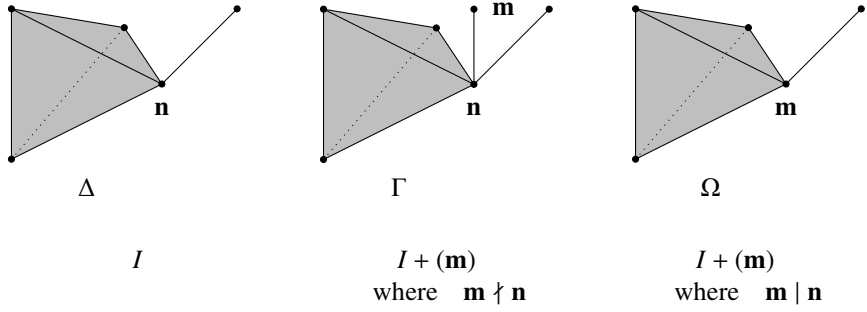
\begin{figure}[!htbp]
$$\begin{array}{c c c}
\begin{tikzpicture}
%\coordinate (X) at (-1,1);
\coordinate (B) at (0, 0);
\coordinate (C) at (2, 1);
\coordinate (D) at (3, 2 );
\coordinate (E) at (0, 2);
\coordinate (F) at (1.5,1.75);
\coordinate (G) at (1,0);
\node[label = below:$\Delta$] at (G) {};
\draw[draw = none, fill=lightgray] (C) -- (E) -- (F);
\draw[draw = none, fill=lightgray] (B) -- (C) -- (E);
\draw[dotted] (B) -- (F);
\draw[-] (E) -- (F);
\draw[-] (C) -- (F);
\draw[-] (B) -- (C);
\draw[-] (B) -- (E);
\draw[-] (C) -- (D);
\draw[-] (C) -- (E);
\draw[black, fill=black] (B) circle(0.04);
\draw[black, fill=black] (C) circle(0.04);
\draw[black, fill=black] (D) circle(0.04);
\draw[black, fill=black] (E) circle(0.04);
\draw[black, fill=black] (F) circle(0.04);
%\node[label = below :$\Delta :$] at (X) {};
\node[label = below :$\bn$] at (C) {};
\end{tikzpicture}
\quad & \quad
\begin{tikzpicture}
%\coordinate (X) at (-1,1);
\coordinate (B) at (0, 0);
\coordinate (C) at (2, 1);
\coordinate (D) at (3, 2 );
\coordinate (E) at (0, 2);
\coordinate (F) at (1.5,1.75);
\coordinate (G) at (1,0);
\node[label = below:$\Gamma$] at (G) {};
\draw[draw = none, fill=lightgray] (C) -- (E) -- (F);
\draw[draw = none, fill=lightgray] (B) -- (C) -- (E);
\draw[dotted] (B) -- (F);
\draw[-] (E) -- (F);
\draw[-] (C) -- (F);
\draw[-] (B) -- (C);
\draw[-] (B) -- (E);
\draw[-] (C) -- (D);
\draw[-] (C) -- (E);
\draw[black, fill=black] (B) circle(0.04);
\draw[black, fill=black] (C) circle(0.04);
\draw[black, fill=black] (D) circle(0.04);
\draw[black, fill=black] (E) circle(0.04);
\draw[black, fill=black] (F) circle(0.04);
%\node[label = below :$\Gamma:$] at (X) {};
\node[label = below :$\bn$] at (C) {};
\coordinate (A) at (2, 2);
\draw[black, fill=black] (A) circle(0.04);
\draw[-] (A) -- (C);
\node[label = right:$\bm$] at (A) {};
\end{tikzpicture}
\quad & \quad 
\begin{tikzpicture}
%\coordinate (X) at (-1,1);
\coordinate (B) at (0, 0);
\coordinate (C) at (2, 1);
\coordinate (D) at (3, 2 );
\coordinate (E) at (0, 2);
\coordinate (F) at (1.5,1.75);
\coordinate (G) at (1,0);
\node[label = below:$\Omega$] at (G) {};
\draw[draw = none, fill=lightgray] (C) -- (E) -- (F);
\draw[draw = none, fill=lightgray] (B) -- (C) -- (E);
\draw[dotted] (B) -- (F);
\draw[-] (E) -- (F);
\draw[-] (C) -- (F);
\draw[-] (B) -- (C);
\draw[-] (B) -- (E);
\draw[-] (C) -- (D);
\draw[-] (C) -- (E);
\draw[black, fill=black] (B) circle(0.04);
\draw[black, fill=black] (C) circle(0.04);
\draw[black, fill=black] (D) circle(0.04);
\draw[black, fill=black] (E) circle(0.04);
\draw[black, fill=black] (F) circle(0.04);
%\node[label = below :$\Omega :$] at (X) {};
\node[label = below :$\bm$] at (C) {};
\end{tikzpicture}\\
&&\\
I & I+(\bm) & I+(\bm)\\
&
 \qwhere \m \nmid \bn & 
 \qwhere \m \mid \bn \\
&&
\end{array}$$
\caption{Three complexes of the same homotopy type}\label{f:intro}
\end{figure}
Now the simplicial complex $\Delta$ in \cref{f:intro} is, via an \say{elementary collapse}, homotopy equivalent to the simplicial complex $\Gamma$ in the same picture, and as a consequence $\Delta$ and $\Gamma$ have the same homology groups. Moreover, $\Delta$ and $\Omega$ are exactly the same complex.

The main question of this paper is the following: 

\begin{question} If we start from a monomial ideal $I=(\bm_1,\ldots,\bm_q,\bn)$
  with a (minimal) free resolution supported on a simplicial complex
  $\Delta$ (an example is $\Delta$ in \cref{f:intro}), can we
  characterize monomials $\bm$ such that $I+(\bm)$ has a (minimal)
  free resolution supported on $\Gamma$? Or monomials $\bm$ such that
  $J=(\bm_1,\ldots,\bm_q,\bm)$ has a (minimal) free resolution
  supported on $\Omega$?
\end{question}

With the monomial $\bn$ chosen as a minimal generator of a monomial ideal $I$ as above, in this paper we define (\cref{d:Cax}) a set of monomials 
$C_I(\bn)$ from which we can pick monomials $\bm$ such that $\Gamma$ or $\Omega$ supports a free resolution of  $I+(\bm)$. More precisely, we prove the following.

\begin{theorem}\label{t:intro}
  Let $I$ be a monomial ideal in a polynomial ring $S$ with at least two variables, let $\bm,\bn$ be monomials with $\bn \in \mingens(I)$.
  Let $\Delta$, $\Omega$ and $\Gamma$ be simplicial complexes as in \cref{f:intro}. 

 \begin{enumerate}
 \item (\cref{p:expansion}) If $\bm \nmid \bn$, then
\smallskip

  \begin{center}
  \begin{tabular}{c}
  $\Gamma$ supports a (minimal) \\
  free resolution of $I+(\bm)$ 
  \end{tabular}
  $\iff$ 
  \begin{tabular}{c}
  $\bm \in C_I(\bn)$ and $\Delta$ supports a (minimal)\\
  free resolution of $I$.
  \end{tabular}
  \end{center} 
 \smallskip

   \item (\cref{c:replacement}) If $\bm \mid \bn$  and $\bm \in C_I(\bn)$, then  
\smallskip

   \begin{center}
 \begin{tabular}{c}
  $\Omega$ supports a (minimal)\\
  free resolution of $I+(\bm)$ 
  \end{tabular}
  $\iff$ 
  \begin{tabular}{c}
  $\Delta$ supports a (minimal)\\
  free resolution of $I$.
  \end{tabular}
    \end{center} 
   \smallskip 
   
   \item (\cref{T:Betti number unchange}) If $\bm \in C_I(\bn)$, for $i \geq 2$ $$\beta_{i,j}(S/I)=\beta_{i,j}(S/(I+(\bm))),$$   and similarly for multigraded betti numbers. Moreover, the first and second betti numbers differ by at most $1$.
   \item (\cref{R: convenientbuild}) If a  monomial $\bv \neq 1$ satisfies $\gcd(\bn,\bv)=1$, then $\beta_i(S/I)=\beta_i(S/(\bv I+(\bn)))$ for all $i\geq 0$.
   \item (\cref{s:last section}) If $\bm \in C_I(\bn)$, the depth, projective dimension, and regularity of $I+(\bm)$ is either equal to or can be computed from that of $I$.

  \item (\cref{t:polar}) If $\P$ denotes the polarization operation, then $\bm\in C_{I}(\bn)  \iff
\mathcal{P}(\bm) \in C_{\mathcal{P}(I)}(\mathcal{P}(\bn)).$

  \item (\cref{c:scm}) If $I$ is sequentially Cohen-Macaulay and $\bm\in C_I(\bn)$, then so is $I+(\bm)$.

 \end{enumerate}  
  \end{theorem}  

As a consequence of \cref{t:intro}~(3) we are able to build, starting from a monomial ideal $I$, infinitely many monomial ideals who share almost the same free resolution. 
 
\begin{theorem}[See \cref{T:Betti number unchanged 2}]\label{T: thm in introduction}
 	Let $I$ be a monomial ideal in the polynomial ring $S=\kappa[x_1,\ldots,x_n]$, $n\geq 2$, over a field $\kappa$, minimally generated by monomials $\bm_1,\ldots,\bm_q$. Suppose $\bv\neq 1$ is a monomial in $S$. Then there are infinitely many monomials 
  $\bm \in S$ such that 
 	 \begin{align*}
 	 	\beta_j(S/(\bv I+(\bm))= \begin{cases}
    \beta_{j}(S/I)+1         & \text{if } j \in \{1,2\} \qand \bm\nmid \bv \bm_i, \forall i\in [q], \\
    \beta_{j}(S/I)           & \text{otherwise.}
      \end{cases}
 \end{align*}	 
 \end{theorem}

In \cref{S: expansion of a labeled simplicial complex}, we present an explicit (and not complex) constructive method to build such monomials $\bm$. Since \cref{T: thm in introduction} works for any monomial ideal in $S$, we can obtain an infinite collection of monomial ideals with predictable betti numbers by applying \cref{T: thm in introduction} repeatedly.
For $I=(a^3b,a^2b^2c,b^4c^2,ac^3)$, the betti numbers of $S/I$ is $(1,4,5,2)$, and the table below demonstrates samples of expansions of $I$ obtained using this method.

\smallskip

\begin{center}
\begin{tabular}[c]{| c | c | c | c |}
 \hline
 (1,4,5,2) & (1,5,6,2) & (1,6,7,2) & (1,7,8,2)\\
 \hline
&&&\\
$(a^3b,a^2b^2c,b^4c^2,ac^3)$ & $(a^4b,a^3b^2c,ab^4c^2,a^2c^3, b^4c^3)$ & $(a^4b^2,a^3b^3c,ab^5c^2,a^2bc^3, b^5c^3,a^6c^3)$& $\cdots$ \\
$(a^4b,a^3b^2c,b^4c^2,a^2c^3)$  & $(a^4bc,a^3b^2c^2,b^4c^3,a^2c^4, a^4b^3)$  & $(a^4b^2c,a^3b^3c^2,b^5c^3,a^2bc^4, a^4b^4,a^3c^5)$ & $\cdots$\\
$(a^3b,a^2b^2c^2,b^4c^3,ac^4)$  & $(a^4b^2,a^3b^3c^2,ab^5c^3,a^2bc^4, a^2c^5)$  & $(a^4b^2,a^3b^3c^2,ab^5c^3,a^2bc^4, a^2c^5,b^6c^4)$ & $\cdots$\\
$(a^3b^2,a^2b^3c,b^5c^2,ac^3)$  &  $(a^4b^2c,a^3b^3c^2,ab^5c^3,a^2c^4, b^5c^3)$  & $(a^4b^2c,a^3b^3c^2,ab^5c^3,a^2c^4, b^5c^3, a^6b^2)$ &$\cdots$ \\
$\vdots$ & $\vdots$ & $\vdots$ & $\vdots$ \\
\hline
\end{tabular} 
\end{center}

\smallskip
	
	Since only $\beta_1$ and $\beta_2$ may increase by one by \cref{T:Betti number unchanged 2}, the betti numbers in this collection will always be in the form of $(1,4+x,5+x,2)$, and the length will not change. In other words, we observe that \cref{T:Betti number unchanged 2} preserves the projective dimension in this example. 
    
We should point out that the idea of studying resolutions of monomial ideals under simplicial collapses of faces of the Stanley-Reisner complex is at the center of Bigdeli and Faridi's paper~\cite{BF},  leading also to new monomials that can be added to a square-free monomial ideal while preserving the betti numbers.

An interesting and powerful application of our construction -- which corresponds topologically to gluing two simplicial complexes along faces -- is forcing homological invariants such as regularity to happen in predetermined ranges of homological degrees (see also Ullery~\cite{U}). This application is in our forthcoming paper~\cite{FL}.

The paper is organized as follows. \cref{s:collapse} introduces simplicial complexes and elementary collapses. \cref{s:resolutions} introduces multigraded free resolutions for monomial ideals. In \cref{S: expansion of a labeled simplicial complex} we start our work to describe monomials $\bm$ that can be added to the generating set of a monomial ideal $I$, so that by adding one edge, a simplicial complex supporting a (minimal) free resolution of $I$ is transformed into a simplicial complex supporting a (minimal) free resolution of $I + (\bm)$. \cref{s:betti} is devoted to a careful analysis of the multigraded betti numbers of expansion ideals, and in \cref{s:algorithm} we show how the results of the earlier sections can be used to create infinitely many monomial ideals with similar betti tables. The regularity and projective dimension of expansion ideals are studied in \cref{s:last section}. In \cref{s:pol} we show that polarization and expansion commute, and use this result to reduce the study of expansions to square-free monomial ideals. Then, in \cref{s:SCM}, we apply Reisner's criterion for Cohen-Macaulay Stanley-Reisner ideals to show that if an ideal is sequentially Cohen-Macaulay, then so are all of its expansions.

\subsubsection*{Acknowledgements} The authors are grateful to  Hal Schenck for pointing them to Ullery's work~\cite{U}.

%%%%%%%%%%%%%%%%%%%%%%%%%%%%%%%%%%%%%%%%%%%%%%%%%%
\section{Simplicial complexes and elementary collapses}\label{s:collapse}
%%%%%%%%%%%%%%%%%%%%%%%%%%%%%%%%%%%%%%%%%%%%%%%%%%

\begin{definition}\label{d:simplicial complex}
Let $V$ be a set. The simplicial complex $\Delta$ over $V$ is a set of
subsets of $V$, such that if $F \in \Delta$ and $G\subseteq F$, then
$G \in \Delta$. An element of $\Delta$ is called a {\bf face} of
$\Delta$. A maximal face of $\Delta$ is called a {\bf facet} of
$\Delta$. If a face is contained in exactly one facet, then
we call it a {\bf free face}. The {\bf dimension} of the face $F$ is
the number of vertices of $F$ minus one, i.e. $\dim(F)=|F|-1$. The dimension of $\Delta$ is the
dimension of its largest face. The
zero dimensional faces of $\Delta$ (i.e. elements of $V$) are called
the {\bf vertices} of $\Delta$. If $V$ is an empty set, then the dimension of $\Delta$ is $-1$. A {\bf simplex} is a simplicial complex
with exactly one facet. If the simplex has $q+1$ vertices, then
it is a {\bf $q$-simplex}.
\end{definition}

Any simplicial complex can be
considered as a collection of simplices (its facets) and their faces.

Now, we can define the notion of a \say{collapse}, which is a way to remove  faces of a simplicial complex while preserving its homotopy type.

\begin{definition} [\cite{K}, Definition 6.13] \label{d:collapse}
Let $\Delta$ be a simplicial complex, $\tau$ be a facet of $\Delta$, and
$\sigma \subsetneq \tau$ be a free face of $\Delta$. A {\bf collapse} of
$\Delta$ along $\sigma$ is the simplicial complex $$ \Delta_{\searrow
  \sigma} = \Delta \ssm \{\gamma \st \sigma \subseteq \gamma\} =
\{F\in \Delta \st \sigma \not \subseteq F \}.$$ If $\dim(\tau) =
\dim(\sigma)+1$, then the collapse is called an {\bf elementary
  collapse}.
\end{definition}

If a simplicial complex $\Gamma$ can be reduced to $\Delta$ with a
sequence of (elementary) collapses, we say $\Gamma$ \say{collapses to} 
$\Delta$.

\begin{proposition}[\cite{K}, Proposition 6.14]\label{p:homotopic}
  If a simplicial complex $\Gamma$ collapses to a simplicial complex
  $\Delta$, then $\Delta$ and $\Gamma$ are homotopy equivalent. In particular, $\rH_i(\Delta)=\rH_i(\Gamma)$ for all $i\geq -1$ where $\rH_i$ is the $i$-th reduced homology group.
\end{proposition}

In our context, we will add a new vertex and an edge on a simplicial
complex $\Delta$ to obtain the simplicial complex $\Gamma$. Then, we
will show $\Gamma$ collapses to $\Delta$, which implies that
their homology groups are isomorphic by \cref{p:homotopic}. 

\begin{figure}[!htbp]
\centering

\tikzset{every picture/.style={line width=0.75pt}} %set default line width to 0.75pt        

\begin{tikzpicture}[x=0.75pt,y=0.75pt,yscale=-1,xscale=1]

%Straight Lines [id:da11903897262945651] 
%Straight Lines [id:da14208390052683084] 
\draw    (175,85) -- (229,145) ;
\draw [shift={(229,145)}, rotate = 61.63] [color={rgb, 255:red, 0; green, 0; blue, 0 }  ][fill={rgb, 255:red, 0; green, 0; blue, 0 }  ][line width=0.75]      (0, 0) circle [x radius= 3.35, y radius= 3.35]   ;
%Straight Lines [id:da689483250802386] 
\draw    (175,85) -- (124,145) ;
\draw [shift={(124,145)}, rotate = 117.02] [color={rgb, 255:red, 0; green, 0; blue, 0 }  ][fill={rgb, 255:red, 0; green, 0; blue, 0 }  ][line width=0.75]      (0, 0) circle [x radius= 3.35, y radius= 3.35]   ;
\draw [shift={(175,85)}, rotate = 117.02] [color={rgb, 255:red, 0; green, 0; blue, 0 }  ][fill={rgb, 255:red, 0; green, 0; blue, 0 }  ][line width=0.75]      (0, 0) circle [x radius= 3.35, y radius= 3.35]   ;
%Straight Lines [id:da5036715595398633] 
\draw    (124,145) -- (229,145) ;
%Straight Lines [id:da5509782105639423] 
\draw    (405.2,83.32) -- (458,145) ;
\draw [shift={(458,145)}, rotate = 62.56] [color={rgb, 255:red, 0; green, 0; blue, 0 }  ][fill={rgb, 255:red, 0; green, 0; blue, 0 }  ][line width=0.75]      (0, 0) circle [x radius= 3.35, y radius= 3.35]   ;
%Straight Lines [id:da11953368008583531] 
\draw    (405.2,83.32) -- (355.33,145) ;
\draw [shift={(355.33,145)}, rotate = 116.12] [color={rgb, 255:red, 0; green, 0; blue, 0 }  ][fill={rgb, 255:red, 0; green, 0; blue, 0 }  ][line width=0.75]      (0, 0) circle [x radius= 3.35, y radius= 3.35]   ;
\draw [shift={(405.2,83.32)}, rotate = 116.12] [color={rgb, 255:red, 0; green, 0; blue, 0 }  ][fill={rgb, 255:red, 0; green, 0; blue, 0 }  ][line width=0.75]      (0, 0) circle [x radius= 3.35, y radius= 3.35]   ;
%Straight Lines [id:da5758721832994516] 
\draw    (355.33,145) -- (458,145) ;

%Straight Lines [id:da6907470682847618] 
\draw    (407,85) -- (482,100) ;
\draw [shift={(482,100)}, rotate = 24.39] [color={rgb, 255:red, 0; green, 0; blue, 0 }  ][fill={rgb, 255:red, 0; green, 0; blue, 0 }  ][line width=0.75]      (0, 0) circle [x radius= 3.35, y radius= 3.35]   ;

% Text Node
\draw (122,126) node [anchor=north west][inner sep=0.75pt]   [align=left] {};
% Text Node
\draw (80,120) node [anchor=north west][inner sep=0.75pt]  [font=\Large]  {$\Delta : $};
% Text Node
\draw (353.26,125.15) node [anchor=north west][inner sep=0.75pt]   [align=left] {};
% Text Node
\draw (300,120) node [anchor=north west][inner sep=0.75pt]  [font=\Large]  {$\Gamma :$};
% Text Node
\draw (380,80) node [anchor=north west][inner sep=0.75pt]   [align=left] {$\displaystyle v_{2}$};
% Text Node
\draw (492,100) node [anchor=north west][inner sep=0.75pt]   [align=left] {$\displaystyle v_{1}$};

\end{tikzpicture}

\caption{An elementary collapse}
\label{F: CAH}
\end{figure}
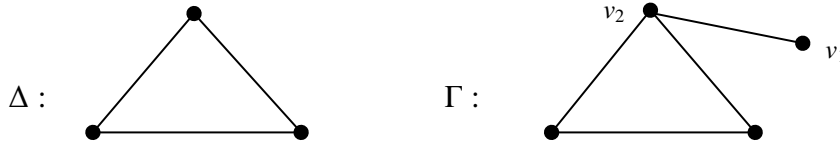 

\begin{example}
Consider $\Gamma$ from \cref{F: CAH}, face $\tau = \{\bv_1\}$ is a free face contained in the facet $\sigma = \{\bv_1,\bv_2\}$. So by \cref{d:collapse}, we can collapse $\tau$ from $\Gamma$ to obtain $\Delta = \Gamma_{\searrow \tau}$. By \cref{p:homotopic}, $\Delta$ and $\Gamma$ are homotopy equivalent, and so their homology modules are isomorphic. It is easy to check:
$$
    \rH_0(\Delta) = \rH_0(\Gamma) = 0,\quad
    \dim\rH_1(\Delta) = \dim \rH_1(\Gamma) = 1,\qand
    \rH_i(\Delta) = \rH_i(\Gamma) = 0,\ i>1.
$$    
\end{example}

%%%%%%%%%%%%%%%%%%%%%%%%%%%%%%%%%%%%%%%%%%%%%%%%%%
\section{Free resolutions}\label{s:resolutions}
%%%%%%%%%%%%%%%%%%%%%%%%%%%%%%%%%%%%%%%%%%%%%%%%%%

Let $S=\kappa[x_1,\ldots ,x_n]$ be a polynomial ring over a field
$\kappa$, and $$\bgx=x_1^{g_1}\cdots x_n^{g_n}$$ be a monomial in $S$. The {\bf degree} and {\bf multidegree} of $\bgx$ are
defined, respectively, as $$\deg(\bgx) = \sum\limits_{j=1}^n g_n \in \NN
\qand \mdeg(\bgx) = \bg=(g_1,\ldots ,g_n) \in \NN^n.$$ To make the
notation simpler, for a monomial $\bgx$ we use $\bgx$ and $\bg$ interchangeably. 
 A \textbf{free resolution} of an $S$-module $S/I$ is a exact sequence of free modules
$$ \FF: 0 \longrightarrow S^{c_p} \stackrel{d_p}{\longrightarrow}
 S^{c_{p-1}} \longrightarrow \cdots \longrightarrow S^{c_1}
 \stackrel{d_1}{\longrightarrow} S^{c_0} \longrightarrow 0$$ such
 that $S/I \cong S^{c_0}/\im(d_1)$. When $d_{i+1}(S^{c_{i+1}})\subseteq (x_1,\ldots,x_n)S^{c_{i}}$ for all $i\geq 0$, $\FF$ is a \textbf{minimal free resolution} of $S/I$ \cite{P}. 
The smallest possible value for the $c_i$ appear in a minimal free resolution are called \textbf{betti numbers}, which are denoted by $\beta_i(S/I)$.

In 1966, Diana Taylor \cite{T} labeled the vertices of a simplex on $q$ vertices with the generators of a monomial ideal $I=(\bm_1,\ldots,\bm_q)$, and each face of the simplex   with the least common multiple ($\lcm$) of its vertex labels. She showed in her thesis that the simplicial chain complex of this simplex can be homogenized into a free resolution (which is usually not minimal) of $I$. This resolution is known as the and named as $\textbf{Taylor resolution}$.  We consider an example.

\begin{example}\label{E: on homogenization}
Let $\Delta$ be a 2-simplex. The chain complex of $\Delta$ is \begin{equation*}
    0\longrightarrow k \xlongrightarrow{\begin{bmatrix} -1 \\ -1 \\ 1 \end{bmatrix}} k^3 \xlongrightarrow{\begin{bmatrix} 0 & -1 & -1 \\ -1 & 1 & 0\\1 & 0 & 1 \end{bmatrix}} k^3 \xlongrightarrow{\begin{bmatrix} 1 & 1 & 1 \end{bmatrix}} k \longrightarrow 0   
\end{equation*}
Let $S=\kappa[x_1,x_2,x_3,x_4]$ for some field $\kappa$. Let $I$ be the monomial ideal generated by monomials $x_1^2x_2x_4,\  x_1x_2^2x_3^2,\ x_1x_2x_3^3$. We label those generators on the vertices of $\Delta$ and label the faces with the lcm of the vertices' labels. (\cref{dia:test}).
\begin{figure}[!ht]
\centering
\tikzset{every picture/.style={line width=0.75pt}} %set default line width to 0.75pt        

\begin{tikzpicture}[x=0.75pt,y=0.75pt,yscale=-1,xscale=1]

%Shape: Triangle [id:dp46954982781265664] 
\draw  [color={rgb, 255:red, 0; green, 0; blue, 0 }  ,draw opacity=1 ][fill={rgb, 255:red, 155; green, 155; blue, 155 }  ,fill opacity=1 ] (179,120) -- (244,217) -- (114,217) -- cycle ;

% Text Node
\draw (156,94) node [anchor=north west][inner sep=0.75pt]   [align=left] {$\displaystyle x_{1}^{2} x_{2} x_{4}$};
% Text Node
\draw (60,217) node [anchor=north west][inner sep=0.75pt]   [align=left] {$\displaystyle x_{1} x_{2}^{2} x_{3}^{2}$};
% Text Node
\draw (246,220) node [anchor=north west][inner sep=0.75pt]   [align=left] {$\displaystyle x_{1} x_{2} x_{3}^{3}$};
% Text Node
\draw (85,146) node [anchor=north west][inner sep=0.75pt]   [align=left] {$\displaystyle x_{1}^{2} x_{2}^{2}x_{3}^{2} x_{4}$};
% Text Node
\draw (215,146) node [anchor=north west][inner sep=0.75pt]   [align=left] {$\displaystyle x_{1}^{2} x_{2} x_{3}^{3} x_{4}$};
% Text Node
\draw (155,220) node [anchor=north west][inner sep=0.75pt]   [align=left] {$\displaystyle x_{1} x_{2}^{2}x_{3}^{3}$};
% Text Node
\draw (153,171) node [anchor=north west][inner sep=0.75pt]   [align=left] {$\displaystyle x_{1}^{2} x_{2}^{2}x_{3}^{3} x_{4}$};

\end{tikzpicture}
\caption{}
\label{dia:test}
\end{figure}

Then, we obtain the free resolution of $S/I$ from the labeled $\Delta$.
\begin{equation*}
    \FF_\Delta: 0 \longrightarrow S \xlongrightarrow{\begin{bmatrix} -x_1x_4 \\ -x_3 \\ x_2 \end{bmatrix}} S^3\xlongrightarrow{\begin{bmatrix} 0 & -x_2x_3^2 & -x_3^3 \\ -x_3 & x_1x_4 & 0\\x_2 & 0 & x_1x_4 \end{bmatrix}} S^3  \xlongrightarrow{\begin{bmatrix} x_1^2x_2x_4 & x_1x_2^2x_3^2 & x_1x_2x_3^3  \end{bmatrix}} S \longrightarrow  0   
\end{equation*}
\end{example}
Since all matrices are not containing non-zero constants, $\FF_\Delta$ is a minimal free resolution. From right to left, the first free non-zero module $\FF_{\Delta_0} = S$ is
augmented. The second free module $\FF_{\Delta_1} =
S^3$ represents the vertices of $\Delta$. The third free module $\FF_{\Delta_2} =
S^3$ is the edges of $\Delta$. The fourth free module $\FF_{\Delta_3} = S$ is the
full triangle. The number of faces equals the rank of the
corresponding free modules.

By this connection, we can write $S$ in terms of the labels of the faces. If we consider the multidegree of the monomial, we will obtain the \textbf{multigraded free resolution}:
$$\FF_\Delta: 0 \longrightarrow S(x_1^2x_2^2x_3^3x_4) \longrightarrow \begin{matrix}
    S(x_1^2x_2^2x_3^2x_4)\\ \oplus \\ S(x_1^2x_2x_3^3x_4) \\ \oplus \\ S(x_1x_2^2x_3^3) 
\end{matrix}
 \longrightarrow \begin{matrix}
     S(x_1^2x_2x_4)\\ \oplus \\S(x_1x_2^2x_3^2) \\ \oplus \\ S(x_1x_2x_3^3) 
 \end{matrix} \longrightarrow S \longrightarrow  0.$$ 
 
Recall that the rank of $S$ in a minimal free resolution is the betti numbers. Accordingly, the number of copies of $S(\bm)$ is the multigraded betti number. The $i$-th \textbf{multigraded betti number} of $S/I$ is defined as \begin{equation*}
    \beta_{i,\bm}(S/I) = \text{ number of copies of } S(\bm) \text{ in } i\text{-th homological degree of } \FF,
\end{equation*}
where $\FF$ is a minimal multigraded free resolution.

Similarly, if considering the degree of monomials, we obtain the \textbf{graded free resolution}:
$$\FF_\Delta: 0 \longrightarrow S(-8) \longrightarrow \begin{matrix}
    S(-7)\\ \oplus \\ S(-7) \\ \oplus \\ S(-6) 
\end{matrix}
 \longrightarrow \begin{matrix}
     S(-4)\\ \oplus \\S(-5 )\\ \oplus \\ S(-5) 
 \end{matrix} \longrightarrow S \longrightarrow  0.$$
And the $i$-th \textbf{graded betti number} is defined as \begin{equation*}
	\beta_{i,j}(S/I) = \text{ number of copies of } S(-j) \text{ in } i\text{-th homological degree of } \FF,
\end{equation*}

  For an monomial ideal $I$, where $\mingens(I)= \{\m_1,\ldots ,
  \m_q\}$, the {\bf lcm lattice}~\cite{GPW} of $I$
  \begin{equation}\label{e:lcm-lattice}
    \LL_I = \{\lcm(\m_{j_1},\ldots ,\m_{j_r}) \st 1\leq j_1<\cdots <j_r\leq q\}
  \end{equation}
is an atomic lattice ordered by divisibility. The least element
$\hat{0}$ = 1 = $\lcm(\emptyset) \in \LL_I$, and the atoms are
$\m_1,\ldots ,\m_q$.

The sum of all $i$-th multigraded betti numbers equals to the $i$-th
betti number. Since all the faces are the $\lcm$ of the labels of some
vertices, we only need to consider the monomials in the lcm
  lattice of $I$
\begin{equation}\label{E:Betti}
    \beta_{i}(S/I) =  \sum_{j\geq 1} \beta_{i,j}(S/I)= \sum_{\m \in \LL_I} \beta_{i,\bm}(S/I).
\end{equation}

The \textbf{projective dimension} of  $S/I$, is  the length of a minimal free resolution
$$\pd(S/I)=\max\{i\mid \beta_i(S/I) \neq 0\}$$
and the \textbf{regularity} of $S/I$ can be thought of as the \say{width} of the minimal free resolution (in the sense of degree)  
	$$\reg(S/I)= \max\{j-i\mid \beta_{i,j}(S/I) \neq 0\}.$$

Bayer and Sturmfels \cite{BS} (see also~\cite{BPS}) generalized the
idea of constructing a free resolution by labeling a $q$-simplex
(Taylor complex) to a labeled simplicial complex. If the simplicial
complex $\Delta$ labeled by $I$ can be ``homogenized" into a free
resolution of $S/I$, then we say that \textbf{$\Delta$ supports a free
  resolution of $S/I$}. If the polynomial ring $S$ is clear, we say
$\Delta$ supports a free resolution of $I$. Before stating the
proposition, we need the following definitions.

\begin{definition} \label{D:subcomplex}
Let $I$ be a monomial ideal and $\bu\in I$ be arbitrary. $\Delta$ is some simplicial complex whose vertices are labeled with the generators of $I$. Then 
\begin{align*} 
\Delta_{\leq \bu} &= \{\sigma \in \Delta \st \lcm (\sigma)\mid \bu\} \nonumber \\ 
\Delta_{< \bu} &= \{\sigma \in \Delta \st \lcm (\sigma) \mid \bu \text{ and } \lcm(\sigma) \neq \bu \} 
\end{align*}
\end{definition}

 \begin{proposition}[\cite{BPS}, {\cite[Proposition 57.5]{P}}]\label{p:supresolution}
 Let $I$ be a monomial ideal and $\Delta$ be a simplicial complex labeled by $I$. Then $\Delta$ supports a free resolution of $I$ if and only if $\Delta_{\leq \bu}$ is acyclic over $\kappa$ for all $1\neq \bu\in \LL_I$. 
\end{proposition}

 This proposition gives us a convenient way to check whether a labeled simplicial complex supports a free resolution of $I$. We need the following proposition to check the minimality of the simplicial resolution. 

 \begin{proposition}[{\cite[Remark 1.4]{BS}}]\label{p:minimality}
     Let $I$ be a monomial ideal and $\Delta$ be a simplicial complex that supports a free resolution of $I$. The resolution is minimal if and only if the label of $\sigma$ and the label of $\tau$ are distinct for any $\sigma\subsetneq \tau \in \Delta$. 
 \end{proposition}
 
The dimensions of the homology groups of the subcomplexes are also related to the betti numbers. The following proposition is the main technique for us to compute the betti numbers in this paper.

\begin{proposition}[\cite{BS}]\label{p:Bayer and Sturmfels}
 Let $\Delta$ be a simplicial complex labeled by the monomial generators of an ideal, $I$. Suppose $\Delta$ supports a free resolution of $I$.
The multigraded betti numbers of $S/I$ are
\begin{equation*}
\beta_{i,\bm}(S/I)=\begin{cases}
    \dim\rH_{i-2}  (\Delta_{<\bm}; \kappa) & \text{if $\Delta_{\leq \bm} \neq \emptyset$}\\
    0 & \text{otherwise},
  \end{cases}    
\end{equation*}
for $i\geq 1$ and monomial $\m$.
\end{proposition}
In most cases, there are a large number of distinct monomials in $I$. We do not expect to check $\Delta_{<\bm}$ for all $\m\in I$. Instead, it is only necessary to check all $1\neq \m \in \LL_I$, as from the same formula, it is easy to see that $\beta_{i,\bm}(S/I)=0$ for all $\m \notin \LL_I$.

%%%%%%%%%%%%%%%%%%%%%%%%%%%%%%%%%%%%%%%%%%%%%%%%%%
\section{Expansion of a labeled simplicial complex} \label{S: expansion of a labeled simplicial complex}
%%%%%%%%%%%%%%%%%%%%%%%%%%%%%%%%%%%%%%%%%%%%%%%%%%

In this section, we will introduce a method to add a vertex to a labeled
simplicial complex, $\Delta$, which supports a (minimal) free resolution of $I$, so that  the new simplicial complex will support a free resolution of $I+(\m)$ for some monomial $\m\notin I$. Our main goal will be to show how one can construct this monomial $\m$. 

For the rest of this section, we assume that $S=\kappa[x_1,\ldots,x_n]$ is a polynomial ring over a field $\kappa$. 

\begin{definition}[\bf{Our basic notation}] \label{d:Cax}
For $\bg =(g_1,\ldots,g_n), \bc=(c_1,\ldots,c_n) \in \NN^n$.
\begin{itemize}
    \item $\bgx=x_1^{g_1}\cdots x_n^{g_n}$ is a monomial in $S$.
    \item the {\bf support of} $\bg$ (or $\bgx$) is defined as $\supp(\bgx)=\supp(\bg)=\{i\in [n]\mid g_i\neq 0\}$.

\item $\bc \geq \bg$ means $c_i \geq g_i$ for each $i \in [n]$ (equivalently $\bgx \mid \bcx$). 
\item If $\bc \geq \bg$,  we define 
$$C_\bg(\bc)= \left \{ \bd \in \NN^n, \bd\neq 0 \st  
             d_i \geq c_i \qforall i \in \supp(\bc -\bg), \          
             d_i\lneq c_i \qforsome i \notin \supp(\bc-\bg)  \right \}.  
$$ 
%Otherwise, when  $\bc$ and $\bg$ do not satisfy those conditions,  we let $C_\bg(\bc)= \emptyset$.

Equivalently we define $$C_{\bgx}(\bcx)=\{\bdx \st \bd \in C_\bg(\bc)\}.$$
\item If $I$ is a monomial ideal in $S$, we denote by $\mingens(I)$ the (unique) minimal monomial generating set of $I$.

\item If $I$ is a monomial ideal in $S$ with monomial generating set $\bm_1,\ldots,\bm_q$, then  
$$\gcd(I)=\gcd(\bm_1,\ldots,\bm_q)=\gcd( \bm \st \bm \in \mingens(I) )
         = \gcd(\bm \st \bm \mbox { monomial in } I).$$

Note that since every monomial generating set of $I$ must contain $\mingens(I)$, 
the definition of $\gcd(I)$ is independent of which generators we pick for $I$.

\item If $I$ is a monomial ideal  in the polynomial ring $S=\kappa[x_1,\ldots,x_n]$ over a field $\kappa$, and $\bn \in I$ is a monomial, then by definition $\gcd(I) \mid \bn$. We use the notation 
$$C_I(\bn)=C_{\gcd(I)}(\bn). $$
When $\gcd(I)= 1$, $C_I(\bn)= \emptyset$ by \cref{l:sizeofC_I}.
\end{itemize}
\end{definition}

\begin{lemma}\label{l:sizeofC_I}
	With notation as in \cref{d:Cax}, let $\bg =(g_1,\ldots,g_n), \bc=(c_1,\ldots,c_n)\in \NN^n$ and $I$ be a monomial ideal in $S$ and $\bn \in \mingens(I)$.
    \begin{enumerate}		
        \item $C_\bg(\bc)\neq \emptyset$ if and only if $n=1$ and $\bc\neq (1)\in \NN^1$, or $n>1$ and $c_i=g_i>0$ for at least one $i\in [n]$.

           \item $C_I(\bn)\neq \emptyset$ if and only if  $n=1$ and  $I \neq (x_1)$, or $n>1$ and $\gcd(I)\neq 1$.

        \item If $C_\bg(\bc)\neq \emptyset$ and $n\geq 2$, then $C_\bg(\bc)$ is an infinite set. 
    
        \item If $C_I(\bn)\neq \emptyset$ and $n\geq 2$, then $C_I(\bn)$ is an infinite set.
     \end{enumerate} 
 \end{lemma}

\begin{proof}
	For~(1), suppose $\bd\in C_\bg(\bc).$ Then $d_i\lneq  c_i$ for some $i\notin \supp(\bc-\bg)$. Since $d_i\in \mathbb{N}$, $\supp(\bc-\bg)\subsetneq \supp(\bc)$. This implies $c_i=g_i> 0$ for at least one $i\in [n]$. Morever, when $n=1$, $\bc\neq (1)$, otherwise $\bd=(d_1)=(0)$.
	
	For the converse, we start with $c_i=g_i> 0$. Then $i\notin \supp(\bc-\bg)$. So, we can choose $d_i\lneq c_i$ since $c_i>0$. By assumption, $\bc\neq (1)$ when $n=1$, so either let $d=(d_1)$ where $d_1\geq 1$ or let $d_j> c_j$ for $j\neq i$. Then $\bd\in C_\bg(\bc)$, which settles~(1). 

    For~(2), if $\bx^\bg=\gcd(I)\neq 1$, then there exist $\bx^\bc=\bn\in \mingens{(I)}$ such that $\bc_i=\bg_i$. So~(1) leads~(2).
    
	If $\bn\geq 2$, there must be at least one $j\neq i$. Since $d_j$ has no upper bound, there are infinitely many different $\bd\in C_\bg(\bc)$. Case of $C_I(\bn)$ is similar.  This settles~(3), and equivalently~(4).

\end{proof}

\begin{lemma}\label{L: miw/v} 
Let $\bc, \bg  \in \NN^n$ be such that $\bc \geq \bg$. With notation as in \cref{d:Cax}, $\bd \in C_\bg(\bc)$ if and only if $\bd=\bc+\ba-\bb$ for some $\ba, \bb \in \NN^n$ such that
\begin{enumerate}
 \item $\bb \neq 0$;
 \item if $\ba =0 $ then  $\bb\neq \bc$;
 \item $\supp(\bb) \cap \supp(\bc -\bg)=\emptyset$;
 \item $\supp(\bb) \cap \supp(\ba) = \emptyset$.
 \end{enumerate}
In particular, for $\bb$, $\bc$ and $\bg$ as above, 
$$\max(g_i, c_i-b_i)=c_i \qforall i \in [n].$$
\end{lemma}

\begin{proof}
   Assume $\bd \in C_\bg(\bc)$, then $\bd=\bc+\ba-\bb$ where for each $i \in [n]$, 
   $$ b_i=\max(c_i-d_i,0) \qand  
      a_i=\max(d_i-c_i,0).$$
   Then by the definition of $C_\bg(\bc)$, 
   since $d_i\lneq c_i$ for some $i \notin \supp(\bc-\bg)$, we must have $b_i >0$, meaning that $\bb \neq 0$. 
   Moreover, if $\ba = 0$,  then since $\bd \neq 0$, we must have $\bb \neq \bc$. Also, from the way $\bb$ and $\ba$ were defined it follows that 
   $$\supp(\bb)\cap \supp(\ba)= \supp(\bb) \cap \supp(\bc-\bg)=\emptyset.$$ 

   Conversely, suppose $\bd=\bc+\ba-\bb \in \NN^n$ for some $\bb, \ba \in \NN^n$ satisfying the four conditions above. Then, since  
   $\supp(\bb)\cap \supp(\ba) =\emptyset$, for each $i \in [n]$ 
   \begin{equation}\label{e:d}
   d_i=\begin{cases}
   c_i+a_i & i \in \supp(\ba)\\
   c_i-b_i & i \in \supp(\bb)\\
   c_i     & \mbox{otherwise}.
   \end{cases}
   \end{equation}
   
   Since $\bb \neq 0$, we have $b_j > 0 $ for some $j \in[n]$. Together with the  fact that  $\supp(\bb) \cap \supp(\bc-\bg)=\emptyset$, from \eqref{e:d} we can conclude that 
 $$ d_j =c_j-b_j \lneq c_j \qforsome j \notin \supp(\bc-\bg).$$  
 Moreover, if $i \in  \supp(\bc -\bg)$, then again by \eqref{e:d} $d_i \geq c_i$.
 Finally, $\bd = 0$ would imply that $\ba=0$, which would then, by assumption, imply that $\bb \neq \bc$, which would then imply that $c_i \neq b_i$ for some $i \in [n]$. Now using \eqref{e:d} this means that $d_i \neq 0$ for some $i \in [n]$, a contradiction. So $\bd \neq 0$, and we have now shown that $\bd \in C_\bg(\bc)$.   
 
 To show the final equality, note that $b_i$ can only be nonzero only when $g_i=c_i$, and so
 $$\max(g_i, c_i-b_i)=
 \begin{cases}
 \max(g_i,c_i)  & \qif c_i \gneq g_i\\
 \max(c_i,c_i-b_i) & \qif c_i=g_i
 \end{cases}= c_i \qforall i \in [n].$$
 \end{proof}

We will use the monomial version of \cref{L: miw/v}, so we give a direct translation below.

\begin{corollary}\label{c:miw/v} 
Let $\bc, \bg  \in \NN^n$ be such that $\bc\geq \bg$, and suppose $\bn=\bcx$ and $\bu=\bgx$. With notation as in \cref{d:Cax}, a monomial $\bm \in C_\bu(\bn)$ if and only if 
$$\bm=\frac{\bn\bw}{\bv}$$ for some monomials $\bw, \bv \in S$ such that
\begin{enumerate}
 \item $\bv \neq 1$;
 \item if $\bw =1 $ then  $\bv\neq \bn$;
 \item $\gcd \Big(\bv, \frac{\bn}{\bu} \Big)=1$ (in particular $\bv \mid \bu$);
 \item $\gcd(\bv,\bw) = 1$ (in particular $\bv \mid \bn$).
 \end{enumerate}
In particular, when the equivalent conditions above hold, $$\lcm \Big (\bu, \frac{\bn}{\bv} \Big )=\bn.$$
\end{corollary}

\begin{example}\label{e:added}
If $\bn=x_1x_2x_3^3$ and $\bu=x_1x_2$, then by picking $\bv=x_1$ and $\bw=x_4$ we have $x_2x_3^3x_4=\frac{x_1x_2x_3^3\cdot x_4}{x_1} \in C_\bu(\bn)$. 
\end{example}

\begin{lemma}
Let $\bm$, $\bn$, and $\bu$ be monomials such that $\bu \mid  \bm \mid \bn$.
Then  
$$C_\bu(\bn)\subset C_\bm(\bn) 
\qand 
C_\bu(\bn)\subset C_\bu(\bm).$$
 \end{lemma}

\begin{proof} Assume $\bm \neq \bn$, otherwise our assertion is trivial.
    Let $\bh\in C_\bu(\bn)$, then $\bh=\frac{\bn\bw}{\bv}$ where $\bv$ and $\bw$ are monomials satisfying the conditions of \cref{c:miw/v}. 
    
    To see $\bh \in  C_\bm(\bn)$, we only need to check Condition~(3) of  \cref{c:miw/v}. 
    $$
    \gcd\left (\bv,\frac{\bn}{\bm} \right )\mid \gcd \left (\bv,\frac{\bn}{\bu} \right ) 
    =1.
    $$

    To see that $\bh \in C_\bu(\bm)$,
    write $\bn=\bm\bm'
    $, and write $\bh=\frac{\bm \bw'}{\bv}$ where $\bw'=\bm'\bw$. To show  $\bh\in C_\bu(\bm)$, we use \cref{c:miw/v} again. Condition~(1) still holds.
    Since $\bm\neq \bn$, $\bm'\neq 1$ and so~(2) holds. For~(3), note that $$
    \gcd\left (\bv,\frac{\bm}{\bu} \right )\mid \gcd \left (\bv,\frac{\bn}{\bu} \right )=1.$$
    Finally,~(4) holds because $$\gcd(\bv,\bw')=\gcd(\bv,\bm') = \gcd \left (\bv, \frac{\bn}{\bm} \right ) \mid \gcd \left (\bv, \frac{\bn}{\bu}\right )=1.$$

\end{proof}

\begin{proposition}\label{P: C_I and "N_I"}
    Let $I$ be a monomial ideal in the polynomial  ring $S$. 
    Suppose $\bm$ and $\bn$ are nontrivial monomials with $\bn \in \mingens(I)$ and $\bm\notin I$. Then  the following statements are equivalent. 
    \begin{enumerate}
		\item $\m\in C_I(\bn)$;
		\item if $\bu\in I$ and $\bm \mid  \bu$,  then $\bn \mid \bu$;  
		\item if $\bu\in \LL_{I+(\m)}$ such that $\bm \mid \bu$ and $\bu\neq \m$, then $\bn\mid \bu$.
  \end{enumerate}
\end{proposition}

\begin{proof}
	$(1)\Longrightarrow(2):$ Suppose $\m\in C_I(\bn)$ and $\bu\in I$ is a monomial such that $\bu \neq \m$ and $\m\mid \bu$. 
        By \cref{c:miw/v}, there are  monomials $\bw$ and $\bv$ such that   
       $$\m=\frac{\bn\bw}{\bv }.$$ 
       		Since $\bu \in I$,  $\bn' \mid \bu$ for some $\bn' \in \mingens(I)$. Then since also $\m\mid \bu$, it follows that $\lcm(\bn',\bm)\mid \bu$. 
		
         As $\bn' \in I$, we can assume $\bn'=\bm'\gcd(I)$ for some monomial $\bm'$. So  $$\lcm(\bn',\bm)= \lcm\Big(\bm'\gcd(I), \frac{\bn \bw}{\bv }\Big).$$ Therefore, using \cref{c:miw/v}, we have
         $$\bn=\lcm\Big(\gcd(I), \frac{\bn}{\bv }\Big) \ \Big | \ 
           \lcm\Big(\bm'\gcd(I), \frac{\bn\bw}{\bv }\Big)=
           \lcm(\bn',\bm)\mid \bu$$ 	

 $(2) \Longrightarrow (3):$ Let  $\bu\in \LL_{I+(\m)}$ with $\bu\neq \m$, if $\m\mid \bu$. Then, since $\bu \neq \m$, for some $\bn' \in \mingens(I)$ we must have $\bn' \mid \bu$, so $\bu \in I$ and therefore by~(2) we have  $\bn\mid \bu$.

$(3) \Longrightarrow (1):$  Suppose the conditions in $(3)$ hold
	for. Our goal is to show that $\m\in
	C_I(\bn)$, and for this purpose we use the equivalent
	conditions in \cref{c:miw/v}.
 
     The monomial $\lcm(\m,\bn)$ is divisible by both $\m$ and $\bn$, and so for some monomials $\bv$ and $\bw$,
    we can write $$\lcm(\m,\bn)=\m \bv=\bn \bw $$ or, equivalently,
    \begin{equation}\label{e:mu} \m=\frac{\bn\bw}{\bv }.
    \end{equation} By cancelling out common factors from $\bv$
    and $\bw$, we can assume that $\gcd(\bv,\bw)=1$.
    
    If $\bv=1$, then $\m=\bn\bw\in I$, which is a contradiction, so $\bv\neq 1$.
    
    If $\bw=1$ and $\bv=\bn$, then $\m=1$ which is a contradiction. 
	
	It remains to show that $\gcd \Big(\bv, \frac{\bn}{\gcd(I)} \Big)=1$.

 	Suppose this is not the case, in other words, suppose there is an $s\in[n]$ such that $x_s\mid \gcd \Big(\bv, \frac{\bn}{\gcd(I)} \Big)$. 
 
  Now $x_s\mid \frac{\bn}{\gcd(I)}$ means that, if $\gamma$ is the highest power such that $x_s ^\gamma \mid \bn$, then there is some other $\bn'$ in $\mingens(I)$ such that 
 \begin{equation}\label{e:mj} x_s^\gamma \nmid \bn'.
 \end{equation}
   Let $\bu=\lcm(\m,\bn') \in \LL_{I+(\m)}$. Then $\bm \mid \bu$ and $\bm \neq \bu$, which implies that $\bn\mid \bu$ by assumption. Therefore, by \eqref{e:mu} 
     \begin{equation}\label{e:mi-mid}
     \bn\mid \lcm\Big(\frac{\bn\bw}{\bv },\bn' \Big).
     \end{equation}
     Now $x_s \mid \bv$ and  we know $\gamma$ is the highest power such that $x_s ^\gamma \mid \bn$, and since $\gcd(\bv,\bw)=1$, $x_s\nmid \bw$, therefore  
 $$x_s^\gamma \nmid \frac{\bn\bw}{\bv }.$$ But this, along with \eqref{e:mj}  
 contradicts \eqref{e:mi-mid}. Therefore $\gcd \Big(\bv, \frac{\bn}{\gcd(I)} \Big)=1$ and we are done.
  \end{proof}

\begin{proposition}[\bf{Minimal generators for $I+(\m)$}]\label{P: m not in I}
     Let $I$ be a monomial ideal in the polynomial  ring $S$. 
    Suppose $\bm$ and $\bn$ are nontrivial monomials with $\bn \in I$ and $\m\in C_I(\bn)$. Then
    \begin{enumerate} 
     \item $\m\notin I$; \label{i:m-notin-I} 
     \item $\m \nmid \bn'$ for every $\bn' \in \mingens{I}\ssm\{\bn\}$; \label{i:nmid} 
       \item  
       if $\bn \notin \mingens(I)$, \label{i:mingens-1}
     then $$\mingens(I+(\m)) =\mingens(I) \cup \{\m\}.$$ 
     \item if $\bn \in \mingens(I)$, then \label{i:mingens-2} 
         $$\mingens(I+(\m)) =
            \begin{cases}
               \mingens(I) \cup \{\m\} & \m \nmid \bn \\
               \big (\mingens(I)\ssm\{\bn\} \big ) \cup \{\m\} & \m \mid \bn. 
            \end{cases}$$
     \end{enumerate}
\end{proposition}

\begin{proof}
	Let $\bu=\gcd(I)$. By \cref{c:miw/v}, for some monomials $\bw$ and $\bv \neq 1$, 
 $$\m=\frac{\bn\bw}{\bv }, \quad 
       \gcd \Big(\bv, \frac{\bn}{\bu} \Big )=1,  \qand 
       \gcd(\bw,\bv)=1.$$

   Since $\bv \neq 1$, $x_s \mid \bv$ for some $s \in [n]$.   Since $\bv \mid \bn$, we must have $x_s \mid \bn$. Since $\gcd (\bv, \frac{\bn}{\bu})=1$, if $\gamma$ is the highest power such that $x_s ^\gamma \mid \bn$, then $x_s ^\gamma \mid \bu=\gcd(I)$, and hence $x_s ^\gamma \mid \bn'$ for every $\bn' \in I$.  
   Therefore if $\m\in I$, then $x_s ^\gamma \mid \bm$. So we have 
   $$ x_s^{\gamma+1} =x_s.x_s^{\gamma} \mid \bv\bm \mid \bn \bw.$$
   On the other hand, since $\gcd(\bw,\bv)=1$, $x_s \nmid \bw$, which implies that  $x_s ^{\gamma+1} \mid \bn$, a contradiction. Therefore $\bm \notin I$, as claimed in \ref{i:m-notin-I}.
    
    Now suppose for some $\bn' \in \mingens(I) \ssm\{\bn\}$ we have $\m \mid \bn'$. Then  \cref{P: C_I and "N_I"} implies that $\bn \mid \bn'$, which since $\bn'$ is minimal generator implies that $\bn = \bn'$, a contradiction. This settles \ref{i:nmid}.
    
    Finally,  statements \ref{i:mingens-1} and \ref{i:mingens-2} follow directly from statements \ref{i:m-notin-I} and \ref{i:nmid}. suppose $\m \nmid \bn$ and $\bn \in \mingens(I)$.   
   \end{proof}

  \begin{proposition}\label{p: latticeiso}
 Let $I$ be a monomial ideal in the polynomial  ring $S$. 
    Suppose $\bm$ and $\bn$ are nontrivial monomials with $\bn \in \mingens(I)$, $\bm\in C_I(\bn)$ and $\bm \mid \bn$.
     Then 
     \begin{enumerate}
         \item\label{i:lcm-equal} if  $\emptyset \neq A\subseteq \mingens(I)\setminus\{\bn\}$ then $\lcm(A\cup(\bn))=\lcm(A\cup(\bm))$;
         \item\label{i:equal}  $\LL_I\setminus \{\bn\}=\LL_{I+(\m)}\setminus \{\m\}$;
         \item \label{i:isom} $\LL_I$ and $\LL_{I+(\m)}$ are isomorphic as lattices.
 \end{enumerate}
 \end{proposition}

\begin{proof}  

To prove \eqref{i:lcm-equal}, suppose $\emptyset \neq A\subseteq \mingens(I)\setminus\{\bn\}$. Let  
$\bu=\lcm(A \cup(\bm)) \in \LL_{I+(\bm)}$. Then, since $A \neq \emptyset$, by \cref{P: m not in I} $\bu \neq \bm$. Since $\bm \mid \bu$, \cref{P: C_I and "N_I"} implies that $\bn \mid \bu$. Therefore, since 
$\bm \mid \bn$, 
$\bu = \lcm(A \cup  (\bm) )=\lcm(A \cup (\bn))$ and we are done.

Now \eqref{i:equal} follows immediately, since by \eqref{i:lcm-equal} the two lattices $\LL_I$ and $\LL_{I+(\bm)}$ maintain the same partial orders, and the only elements that are different in the two lattices are those corresponding to the single vertices $\bm$ and $\bn$. Finally \eqref{i:isom} is via the isomorphism 
$$f: \LL_I \longrightarrow \LL_{I+(\bm)}$$ which takes $\bn$ to $\bm$ and every other element of $\LL_I$ to itself.
	\end{proof}

 \begin{theorem}[\textbf{Main Theorem 1}]\label{c:replacement}
  Let $I$ be a monomial ideal in the polynomial ring $S=\kappa[x_1,\ldots ,x_n]$, and let $\bn \in \mingens(I)$, $\bm \in C_I(\bn)$ such that $\bm \mid \bn$. Let $\Delta$ be a simplicial complex whose vertices are labeled with the generators of $I$, and let $\Gamma$ be the same as $\Delta$, except that the vertex of $\Delta$ labeled with $\{\bn\}$ is labeled with  $\{\bm\}$ in $\Gamma$. 
  \begin{enumerate}
      \item For every monomial $\bu \in \LL_I\setminus \{\bn\}=\LL_{I+(\m)}\setminus \{\m\}$, $\Delta_{\leq \bu}=\Gamma_{\leq \bu}$ and $\Delta_{< \bu}=\Gamma_{< \bu}$.
      \item $\Delta$ supports a (minimal) free resolution of $I$ if and only if  $\Gamma$ supports a (minimal) free resolution of $I+(\bm)$.
  \end{enumerate}
  \end{theorem}  

\begin{proof} 
(1): Choose $\bu\in \LL_{I+(\m)}\setminus \{\m\}=\LL_{I}\setminus \{\bn\}$ (\cref{p: latticeiso}(2)). By \cref{P: C_I and "N_I"} and the fact that $\bm \mid \bn$, we have $$\m\mid \bu \iff  \bn\mid \bu.$$ 
Moreover, by \cref{p: latticeiso}~(1) and our labeling of $\Gamma$, 
if $\sigma$ is a face of $\Gamma=\Delta$, then $\sigma$ will have the same monomial label in each complex. It follows that
 $$\sigma \in \Gamma_{\leq \bu}  \iff \sigma \in \Delta_{\leq \bu} \qand  \sigma \in \Gamma_{< \bu}  \iff \sigma \in \Delta_{< \bu}.$$ 
Hence, $\Gamma_{\leq \bu}=\Delta_{\leq \bu}$ and $\Gamma_{< \bu}=\Delta_{< \bu}$ for all $\bu\in \LL_{I+(\m)}\setminus \{\bm\}$.

(2): By~(1) for all monomials $\bu$ in the lcm lattices of $I$ and $I+(\bm)$, besides $\bm$ and $\bn$,  we have $\Delta_{\leq \bu}$ is acyclic if and only if  
$\Gamma_{\leq \bu}$ is acyclic. Moreover $\Delta_{\leq \bn}$ and $\Gamma_{\leq \bm}$ are both single points, and hence always acyclic.  It follows that $\Delta$ supports a free resolution of $I$ if and only if $\Gamma$ supports a free resolution of $I+(\bm)$ by \cref{p:supresolution}. 

By \cref{p: latticeiso}~(1), the monomial label of any $\tau\in \Delta$ is the same as the monomial label of $\tau \in \Gamma$ except for the vertex labeling by $\bn$. It follows that the free resolution of $\Delta$ is minimal if and only if the free resolution of $\Gamma$ is minimal by \cref{p:minimality}.
\end{proof}

\cref{c:replacement} fails without the condition $\bm \mid \bn$; the following is an example of such a case.

\begin{example}\label{e:IJJ'} Let $I=(abc,cde,ace)$. Then $ae, a^2e \in C_I(ace)$. Let $J=(abc,cde,a^2e)$ and $J'=(abc,cde,ae)$. Then 
%\begin{figure}[!htbp]
$$\begin{array}{c|c|c}
\mbox{supports resolution of } I&
\mbox{does not support resolution of } J&
\mbox{supports resolution of } J'\\
\hline
&&\\
\begin{tikzpicture}[label distance=-5pt] 
\coordinate (A) at (0,0);
\coordinate (B) at (1,0);
\coordinate (C) at (2,0);
%\coordinate (D) at (1,  1);
\draw[black, fill=black] (A) circle(0.04);
\draw[black, fill=black] (B) circle(0.04);
\draw[black, fill=black] (C) circle(0.04);
%\draw[black, fill=black] (D) circle(0.04);
\draw[-] (A) -- (B);
\draw[-] (B) -- (C);
%\draw[-] (B) -- (D);
\node[label = below :$abc$] at (A) {};
\node[label = above :$ace$] at (B) {};
\node[label =  below :$cde$] at (C) {};
%\node[label = above :$$] at (D) {};
\end{tikzpicture} &
\begin{tikzpicture}[label distance=-5pt] 
\coordinate (A) at (0, 0);
\coordinate (B) at (1,0);
\coordinate (C) at (2,0);
%\coordinate (D) at (1,  1);
\draw[black, fill=black] (A) circle(0.04);
\draw[black, fill=black] (B) circle(0.04);
\draw[black, fill=black] (C) circle(0.04);
%\draw[black, fill=black] (D) circle(0.04);
\draw[-] (A) -- (B);
\draw[-] (B) -- (C);
%\draw[-] (B) -- (D);
\node[label = below :$abc$] at (A) {};
\node[label = above:$a^2e$] at (B) {};
\node[label =  below :$cde$] at (C) {};
%\node[label = above :$a^2e$] at (D) {};
\end{tikzpicture} &
\begin{tikzpicture}[label distance=-5pt] 
\coordinate (A) at (0, 0);
\coordinate (B) at (1,0);
\coordinate (C) at (2,0);
%\coordinate (D) at (1,  1);
\draw[black, fill=black] (A) circle(0.04);
\draw[black, fill=black] (B) circle(0.04);
\draw[black, fill=black] (C) circle(0.04);
%\draw[black, fill=black] (D) circle(0.04);
\draw[-] (A) -- (B);
\draw[-] (B) -- (C);
%\draw[-] (B) -- (D);
\node[label = below :$abc$] at (A) {};
\node[label = above: $ae$] at (B) {};
\node[label = below: $cde$] at (C) {};
%\node[label = above: $ae$] at (D) {};
\end{tikzpicture}\\
&\quad a^2e \nmid ace & ae \mid ace\\
&&\mbox{(\cref{c:replacement})}
\end{array}
$$
%\caption{}\label{f:}
%\end{figure}

\end{example}

\begin{theorem} [\textbf{Main Theorem 2}] \label{p:expansion}
  Let $I$ be a monomial ideal in the polynomial ring $S=\kappa[x_1,\ldots ,x_n]$, and $\Delta$ a simplicial complex
  supporting a free resolution of $I$. Let $\bn$ and $\bm$ be monomials with $\bn \in \mingens(I)$, $\bm \notin I$,  and $\bm \nmid \bn$.  Suppose  $\Gamma$ is an expansion of $\Delta$ by adding a new edge attached to $\bn$, with a new vertex labeled with $\bm$, i.e. $$\Gamma = \Delta \cup \{\{\m\}, \{\m, \bn\}\}.$$
   Then $\Gamma$ supports a  free resolution of $I+(\bm)$ if and only if $\bm \in C_I(\bn)$.

Moreover, when $\bm \in C_I(\bn)$, the resolution of $I$ supported on $\Delta$ is minimal if and only if 
 the resolution of $I+(\bm)$ supported on $\Gamma$ is minimal.
\end{theorem}

\begin{proof}
$(\Longleftarrow):$
By \cref{p:supresolution} we need to check $\Gamma_{\leq \bu}$ is acyclic for all $1 \neq \bu\in \LL_{I+(\m)}$. 
 There are two scenarios.

\begin{itemize}
\item If $\m \nmid \bu$, then $\lcm(\m,\bn) \nmid \bu$,  so $\Gamma_{\leq \bu}$ =
$\Delta_{\leq \bu}$. Since $\Delta$ supports a free resolution of $I$,
$\Gamma_{\leq \bu}$ = $\Delta_{\leq \bu}$ is acyclic and we are done.

\item If $\bu = \m$, then since $\m\notin I$, $\bn' \nmid \m$ for all $\bn' \in \mingens I$. So $\Gamma_{\leq \bm}=\{\m\}$ is acyclic.

\item If $\m\mid \bu$ and $\bu \neq \m$, then $\bn \mid \bu$ by \cref{P: C_I and "N_I"}. Therefore $\lcm(\m,\bn)\mid \bu$, and $$\Gamma_{\leq \bu} =
\Delta_{\leq \bu} \cup \{{\m}, \{\m,\bn\}\}.$$ Since $\bn \in \Delta_{\leq
  u}$ and $\m$ is a free vertex,
$$\Delta_{\leq \bu} = {\Gamma_{\leq \bu}}_{\searrow \{\m\}}.$$ Hence,
$\Gamma_{\leq \bu}$ and $\Delta_{\leq \bu}$ are homotopy equivalent, and
since $\Delta_{\leq \bu}$ is acyclic, so is $\Gamma_{\leq \bu}$.
\end{itemize}

We have shown that for all $1\neq \bu \in \LL_{I+(\m)}$, $\Gamma_{\leq \bu}$
is acyclic. Therefore, by \cref{p:supresolution}, $\Gamma$ supports a free resolution of $I+(\m)$.

$(\Longrightarrow):$  Suppose $\m\notin C_I(\bn)$. Then, by \cref{P: C_I and "N_I"}, there exist $\bv\in \LL_{I+(\bm)}$ such that  $\bv \neq \bm$, 
$\m\mid \bv$ and $\bn\nmid \bv$. Then the edge $\{\bm,\bn\}$ is not contained in $\Gamma_{\leq \bv}$ and $\bm \in \Gamma_{\leq \bv}$. Therefore 
$\Gamma_{\leq \bv}$ contains an isolated vertex $\bm$ and at least one more vertex (since $\bv\neq \bm$) and is hence disconnected, and in particular not acyclic. So $\Gamma$ cannot support a free resolution of $I+(\m)$ by \cref{p:supresolution};  a contraction.

Finally, suppose the resolution of $I$ supported on $\Delta$ is minimal, and suppose $\emptyset \neq \tau \subsetneq \sigma \in \Gamma$. If $\sigma, \tau \in \Delta$, then by \cref{p:minimality} they have distinct monomial labels.  It remains to consider the faces of $\Gamma$ which are not in $\Delta$, which are exactly $\{\m\}$ and $\{\m,\bn\}$. Since $\m\nmid \bn$, we know $\bn\neq \lcm(\m,\bn)\neq \bm$, and now  from \cref{p:minimality} it follows that $\Gamma$ supports a minimal free resolution of $I+(\m)$. The converse follows immediately as $\Delta$ is a subcomplex of $\Gamma$.

\end{proof}

\begin{example}\label{e:IJJ'-1} Let $I=(abc,cde,ace)$ be as in \cref{e:IJJ'} with a resolution supported on a path, and take $a^2e \in C_I(ace)$, and let $J=I+(a^2e)$. Then by \cref{p:expansion}, we have  
$$\begin{array}{c|c}
\mbox{supports resolution of } I&
\mbox{supports resolution of } J\\
\hline
&\\
\begin{tikzpicture}[label distance=-5pt] 
\coordinate (A) at (0, 0);
\coordinate (B) at (1,0);
\coordinate (C) at (2,0);
%\coordinate (D) at (1,  1);
\draw[black, fill=black] (A) circle(0.04);
\draw[black, fill=black] (B) circle(0.04);
\draw[black, fill=black] (C) circle(0.04);
%\draw[black, fill=black] (D) circle(0.04);
\draw[-] (A) -- (B);
\draw[-] (B) -- (C);
%\draw[-] (B) -- (D);
\node[label = below :$abc$] at (A) {};
\node[label = below 
:$ace$] at (B) {};
\node[label =  below :$cde$] at (C) {};
%\node[label = above :$$] at (D) {};
\end{tikzpicture} &
\begin{tikzpicture}[label distance=-5pt] 
\coordinate (A) at (0, 0);
\coordinate (B) at (1,0);
\coordinate (C) at (2,0);
\coordinate (D) at (1,  1);
\draw[black, fill=black] (A) circle(0.04);
\draw[black, fill=black] (B) circle(0.04);
\draw[black, fill=black] (C) circle(0.04);
\draw[black, fill=black] (D) circle(0.04);
\draw[-] (A) -- (B);
\draw[-] (B) -- (C);
\draw[-] (B) -- (D);
\node[label = below :$abc$] at (A) {};
\node[label = below:$ace$] at (B) {};
\node[label =  below :$cde$] at (C) {};
\node[label = above :$a^2e$] at (D) {};
\end{tikzpicture}
\end{array}$$
\end{example}

\begin{example}\label{R: 4}
Continuing with \cref{E: on homogenization}, we are going to expand
$\Delta$ (\cref{dia:test}) to show \cref{p:expansion} and \cref{c:replacement}.
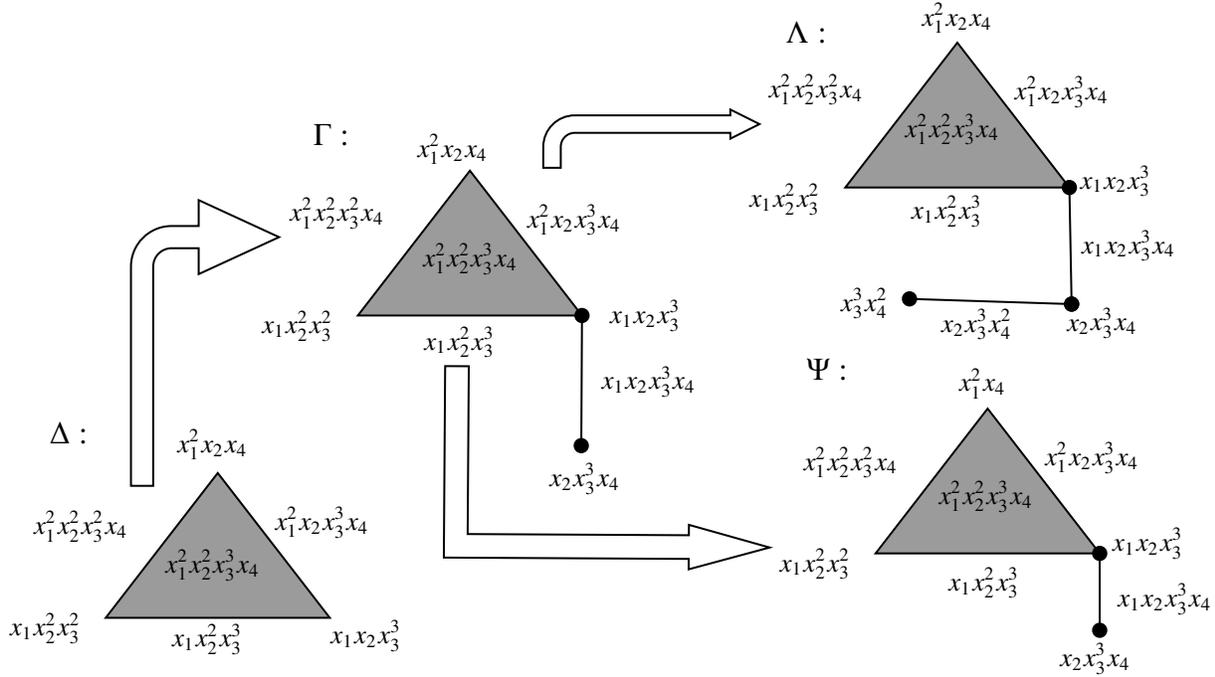
\begin{figure}[ht!]
    \centering

\tikzset{every picture/.style={line width=0.75pt}} %set default line width to 0.75pt        

\begin{tikzpicture}[x=0.75pt,y=0.75pt,yscale=-1,xscale=1]
%uncomment if require: \path (0,617); %set diagram left start at 0, and has height of 617

%Shape: Triangle [id:dp7555840552956326] 
\draw  [color={rgb, 255:red, 0; green, 0; blue, 0 }  ,draw opacity=1 ][fill={rgb, 255:red, 155; green, 155; blue, 155 }  ,fill opacity=1 ] (501.32,58.41) -- (557.83,131.51) -- (444.81,131.51) -- cycle ;
%Straight Lines [id:da5051729077043039] 
\draw    (557.83,131.51) -- (558.95,190.36) ;
\draw [shift={(558.95,190.36)}, rotate = 88.91] [color={rgb, 255:red, 0; green, 0; blue, 0 }  ][fill={rgb, 255:red, 0; green, 0; blue, 0 }  ][line width=0.75]      (0, 0) circle [x radius= 3.35, y radius= 3.35]   ;
\draw [shift={(557.83,131.51)}, rotate = 88.91] [color={rgb, 255:red, 0; green, 0; blue, 0 }  ][fill={rgb, 255:red, 0; green, 0; blue, 0 }  ][line width=0.75]      (0, 0) circle [x radius= 3.35, y radius= 3.35]   ;
%Straight Lines [id:da46540465536614506] 
\draw    (558.95,190.36) -- (477.19,187.72) ;
\draw [shift={(477.19,187.72)}, rotate = 181.85] [color={rgb, 255:red, 0; green, 0; blue, 0 }  ][fill={rgb, 255:red, 0; green, 0; blue, 0 }  ][line width=0.75]      (0, 0) circle [x radius= 3.35, y radius= 3.35]   ;
\draw [shift={(558.95,190.36)}, rotate = 181.85] [color={rgb, 255:red, 0; green, 0; blue, 0 }  ][fill={rgb, 255:red, 0; green, 0; blue, 0 }  ][line width=0.75]      (0, 0) circle [x radius= 3.35, y radius= 3.35]   ;
%Shape: Triangle [id:dp1582357404185517] 
\draw  [color={rgb, 255:red, 0; green, 0; blue, 0 }  ,draw opacity=1 ][fill={rgb, 255:red, 155; green, 155; blue, 155 }  ,fill opacity=1 ] (255.49,123.02) -- (312.01,196.12) -- (198.98,196.12) -- cycle ;
%Straight Lines [id:da07701017817670341] 
\draw    (312.01,196.12) -- (311.65,261.8) ;
\draw [shift={(311.65,261.8)}, rotate = 90.31] [color={rgb, 255:red, 0; green, 0; blue, 0 }  ][fill={rgb, 255:red, 0; green, 0; blue, 0 }  ][line width=0.75]      (0, 0) circle [x radius= 3.35, y radius= 3.35]   ;
\draw [shift={(312.01,196.12)}, rotate = 90.31] [color={rgb, 255:red, 0; green, 0; blue, 0 }  ][fill={rgb, 255:red, 0; green, 0; blue, 0 }  ][line width=0.75]      (0, 0) circle [x radius= 3.35, y radius= 3.35]   ;
%Bend Arrow [id:dp3799217634273553] 
\draw   (292.48,121.57) -- (292.48,110.98) .. controls (292.48,102.21) and (299.58,95.1) .. (308.35,95.1) -- (386.84,95.1) -- (386.84,92.16) -- (401.54,99.51) -- (386.84,106.86) -- (386.84,103.92) -- (308.35,103.92) .. controls (304.46,103.92) and (301.3,107.08) .. (301.3,110.98) -- (301.3,121.57) -- cycle ;
%Shape: Triangle [id:dp3041274264603606] 
\draw  [color={rgb, 255:red, 0; green, 0; blue, 0 }  ,draw opacity=1 ][fill={rgb, 255:red, 155; green, 155; blue, 155 }  ,fill opacity=1 ] (516.6,243.1) -- (573.11,316.19) -- (460.09,316.19) -- cycle ;
%Straight Lines [id:da2962302271316941] 
\draw    (573.11,316.19) -- (573.11,355) ;
\draw [shift={(573.11,355)}, rotate = 90] [color={rgb, 255:red, 0; green, 0; blue, 0 }  ][fill={rgb, 255:red, 0; green, 0; blue, 0 }  ][line width=0.75]      (0, 0) circle [x radius= 3.35, y radius= 3.35]   ;
\draw [shift={(573.11,316.19)}, rotate = 90] [color={rgb, 255:red, 0; green, 0; blue, 0 }  ][fill={rgb, 255:red, 0; green, 0; blue, 0 }  ][line width=0.75]      (0, 0) circle [x radius= 3.35, y radius= 3.35]   ;

%Bend Up Arrow [id:dp2693686885718718] 
\draw   (254.93,221.73) -- (254.35,306.43) -- (365.92,307.03) -- (365.95,301.78) -- (406.63,313.22) -- (365.8,324.22) -- (365.83,318.97) -- (242.34,318.3) -- (243,221.66) -- cycle ;
%Shape: Triangle [id:dp8218749328639026] 
\draw  [color={rgb, 255:red, 0; green, 0; blue, 0 }  ,draw opacity=1 ][fill={rgb, 255:red, 155; green, 155; blue, 155 }  ,fill opacity=1 ] (128.31,275.61) -- (184.82,348.71) -- (71.8,348.71) -- cycle ;
%Bend Arrow [id:dp07210125897777098] 
\draw   (84.54,282.09) -- (84.54,171.44) .. controls (84.54,160.13) and (93.72,150.95) .. (105.03,150.95) -- (118.69,150.95) -- (118.69,137.88) -- (159.24,156.56) -- (118.69,175.23) -- (118.69,162.16) -- (105.03,162.16) .. controls (99.9,162.16) and (95.75,166.32) .. (95.75,171.44) -- (95.75,282.09) -- cycle ;

% Text Node
\draw (41.97,248.05) node [anchor=north west][inner sep=0.75pt]  [font=\Large] [align=left] {$\displaystyle {\displaystyle \Delta :}$};
% Text Node
\draw (105.92,251.8) node [anchor=north west][inner sep=0.75pt]   [align=left] {$\displaystyle x_{1}^{2} x_{2} x_{4}$};
% Text Node
\draw (21.52,344.49) node [anchor=north west][inner sep=0.75pt]   [align=left] {$\displaystyle x_{1} x_{2}^{2} x_{3}^{2}$};
% Text Node
\draw (183.3,348.26) node [anchor=north west][inner sep=0.75pt]   [align=left] {$\displaystyle x_{1} x_{2} x_{3}^{3}$};
% Text Node
\draw (33.52,293.25) node [anchor=north west][inner sep=0.75pt]   [align=left] {$\displaystyle x_{1}^{2} x_{2}^{2}$$\displaystyle x_{3}^{2} x_{4}$};
% Text Node
\draw (155.24,290.98) node [anchor=north west][inner sep=0.75pt]   [align=left] {$\displaystyle x_{1}^{2} x_{2}$$\displaystyle x_{3}^{3} x_{4}$};
% Text Node
\draw (103.25,350) node [anchor=north west][inner sep=0.75pt]   [align=left] {$\displaystyle x_{1} x_{2}^{2}$$\displaystyle x_{3}^{3}$};
% Text Node
\draw (100,312.59) node [anchor=north west][inner sep=0.75pt]   [align=left] {$\displaystyle x_{1}^{2} x_{2}^{2}$$\displaystyle x_{3}^{3} x_{4}$};
% Text Node
\draw (175.19,98.1) node [anchor=north west][inner sep=0.75pt]  [font=\Large] [align=left] {$\displaystyle {\displaystyle \Gamma :}$};
% Text Node
\draw (227.04,102.74) node [anchor=north west][inner sep=0.75pt]   [align=left] {$\displaystyle x_{1}^{2} x_{2} x_{4}$};
% Text Node
\draw (148.7,191.9) node [anchor=north west][inner sep=0.75pt]   [align=left] {$\displaystyle x_{1} x_{2}^{2} x_{3}^{2}$};
% Text Node
\draw (324.15,186.61) node [anchor=north west][inner sep=0.75pt]   [align=left] {$\displaystyle x_{1} x_{2} x_{3}^{3}$};
% Text Node
\draw (162.72,135.12) node [anchor=north west][inner sep=0.75pt]   [align=left] {$\displaystyle x_{1}^{2} x_{2}^{2}$$\displaystyle x_{3}^{2} x_{4}$};
% Text Node
\draw (282.42,138.4) node [anchor=north west][inner sep=0.75pt]   [align=left] {$\displaystyle x_{1}^{2} x_{2}$$\displaystyle x_{3}^{3} x_{4}$};
% Text Node
\draw (230.43,200) node [anchor=north west][inner sep=0.75pt]   [align=left] {$\displaystyle x_{1} x_{2}^{2}$$\displaystyle x_{3}^{3}$};
% Text Node
\draw (230,157.99) node [anchor=north west][inner sep=0.75pt]   [align=left] {$\displaystyle x_{1}^{2} x_{2}^{2}$$\displaystyle x_{3}^{3} x_{4}$};
% Text Node
\draw (293.62,269.18) node [anchor=north west][inner sep=0.75pt]  [rotate=-359.82] [align=left] {$\displaystyle x_{2} x_{3}^{3} x_{4}$};
% Text Node
\draw (320,220) node [anchor=north west][inner sep=0.75pt]   [align=left] {$\displaystyle x_{1} x_{2} x_{3}^{3} x_{4}$};
% Text Node
\draw (551.08,359.81) node [anchor=north west][inner sep=0.75pt]  [rotate=-359.82] [align=left] {$\displaystyle x_{2} x_{3}^{3} x_{4}$};
% Text Node
\draw (500.19,221.05) node [anchor=north west][inner sep=0.75pt]   [align=left] {$\displaystyle x_{1}^{2} x_{4}$};
% Text Node
\draw (409.81,311.97) node [anchor=north west][inner sep=0.75pt]   [align=left] {$\displaystyle x_{1} x_{2}^{2} x_{3}^{2}$};
% Text Node
\draw (577.64,301.63) node [anchor=north west][inner sep=0.75pt]   [align=left] {$\displaystyle x_{1} x_{2} x_{3}^{3}$};
% Text Node
\draw (421.8,260.73) node [anchor=north west][inner sep=0.75pt]   [align=left] {$\displaystyle x_{1}^{2} x_{2}^{2}$$\displaystyle x_{3}^{2} x_{4}$};
% Text Node
\draw (543.52,258.47) node [anchor=north west][inner sep=0.75pt]   [align=left] {$\displaystyle x_{1}^{2} x_{2}$$\displaystyle x_{3}^{3} x_{4}$};
% Text Node
\draw (495,321.55) node [anchor=north west][inner sep=0.75pt]   [align=left] {$\displaystyle x_{1} x_{2}^{2}$$\displaystyle x_{3}^{3}$};
% Text Node
\draw (490,278.06) node [anchor=north west][inner sep=0.75pt]   [align=left] {$\displaystyle x_{1}^{2} x_{2}^{2}$$\displaystyle x_{3}^{3} x_{4}$};
% Text Node
\draw (580.27,329.88) node [anchor=north west][inner sep=0.75pt]   [align=left] {$\displaystyle x_{1} x_{2} x_{3}^{3} x_{4}$};
% Text Node
\draw (423.87,216.13) node [anchor=north west][inner sep=0.75pt]  [font=\Large] [align=left] {$\displaystyle {\displaystyle \Psi :}$};
% Text Node
\draw (481.46,36.48) node [anchor=north west][inner sep=0.75pt]   [align=left] {$\displaystyle x_{1}^{2} x_{2} x_{4}$};
% Text Node
\draw (413.42,46.42) node [anchor=north west][inner sep=0.75pt]  [font=\Large] [align=left] {$\displaystyle {\displaystyle \Lambda :}$};
% Text Node
\draw (394.53,128.04) node [anchor=north west][inner sep=0.75pt]   [align=left] {$\displaystyle x_{1} x_{2}^{2} x_{3}^{2}$};
% Text Node
\draw (561.35,118.71) node [anchor=north west][inner sep=0.75pt]   [align=left] {$\displaystyle x_{1} x_{2} x_{3}^{3}$};
% Text Node
\draw (404.24,72.78) node [anchor=north west][inner sep=0.75pt]   [align=left] {$\displaystyle x_{1}^{2} x_{2}^{2}$$\displaystyle x_{3}^{2} x_{4}$};
% Text Node
\draw (528.24,73.79) node [anchor=north west][inner sep=0.75pt]   [align=left] {$\displaystyle x_{1}^{2} x_{2}$$\displaystyle x_{3}^{3} x_{4}$};
% Text Node
\draw (476.25,134.08) node [anchor=north west][inner sep=0.75pt]   [align=left] {$\displaystyle x_{1} x_{2}^{2}$$\displaystyle x_{3}^{3}$};
% Text Node
\draw (473,92.63) node [anchor=north west][inner sep=0.75pt]   [align=left] {$\displaystyle x_{1}^{2} x_{2}^{2}$$\displaystyle x_{3}^{3} x_{4}$};
% Text Node
\draw (554.58,190.46) node [anchor=north west][inner sep=0.75pt]  [rotate=-359.82] [align=left] {$\displaystyle x_{2} x_{3}^{3} x_{4}$};
% Text Node
\draw (561.97,151.37) node [anchor=north west][inner sep=0.75pt]   [align=left] {$\displaystyle x_{1} x_{2} x_{3}^{3} x_{4}$};
% Text Node
\draw (440.53,180.55) node [anchor=north west][inner sep=0.75pt]  [rotate=-359.82] [align=left] {$\displaystyle x_{3}^{3} x_{4}^{2}$};
% Text Node
\draw (491.24,191.07) node [anchor=north west][inner sep=0.75pt]   [align=left] {$\displaystyle x_{2} x_{3}^{3} x_{4}^{2}$};

\end{tikzpicture}
\caption{Expand on $\Delta$, related to \cref{R: 4}}
\label{expand secondly}
\end{figure}
Let $$\m_\gamma = x_2x_3^3x_4 \in C_I(x_1x_2x_3^3),\quad  
\m_\lambda = x_3^3x_4^2 \in C_{I+(\m_\gamma)}(x_2x_3^3x_4), \quad  
\m_\psi = x_1^2x_4 \in C_{I+(\m_\gamma)}(x_1^2x_2x_4).$$ 
Following \cref{expand secondly}, by \cref{p:expansion} and \cref{c:replacement} we have  
$$\begin{array}{lll}
\Gamma &= \Delta \cup \{\m_\gamma, \{\m_\gamma,x_1x_2x_3^3\}\} & \text{ supports a free resolution of } I+(\m_\gamma), \\
\Lambda &= \Gamma \cup \{\m_\lambda, \{\m_\lambda,\m_\gamma\}\} & \text{ supports a free resolution of } I+(\m_\gamma, \m_\lambda),\\
\Psi &= \Gamma  & \text{ supports a free resolution of } I+(\m_\gamma, \m_\psi).
\end{array}$$
\end{example}

%%%%%%%%%%%%%%%%%%%%%%%%%%%%%%%%%%%%%%%%%%%%%%%%%%
\section{Applications to resolutions of expansion ideals}\label{s:betti}
%%%%%%%%%%%%%%%%%%%%%%%%%%%%%%%%%%%%%%%%%%%%%%%%%%

We are now ready to explore the consequence of \cref{p:expansion}. In this section we show that \cref{p:expansion} allows us to \say{grow} or \say{expand} an ideal by one generator, while controlling its betti numbers, projective dimension, and regularity. Later on, in \cref{s:algorithm}, we will show that this process can be repeated again and again to build infinitely many ideals with controlled homological invariants.

\begin{theorem}[\bf{betti numbers of $I+(\m)$}] \label{T:Betti number unchange}
  Let $I$ be a monomial ideal   in the polynomial ring $S=\kappa[x_1,\ldots,x_n]$ with $\gcd(I) \neq 1$. 
  Let $\bn \in \mingens(I)$, $\m\in C_I(\bn)$, and $j >0$. Then for all $1\neq \bu\in \LL_{I+(\m)}$, we have the following relations between the multigraded betti numbers of $S/I$ and $S/(I+(\m))$
$$
\beta_{j,\bu}(S/(I+(\m))) =
  \begin{cases}
   1        & j=1, \quad \bu=\m  \\
   0        & j=1, \quad \bu=\bn, \quad \m \mid \bn\\
   \beta_{j,\bu}(S/I)+1        & j=2, \quad \bu=\lcm(\m,\bn), \quad \m \nmid \bn \\
   \beta_{j,\bu}(S/I)           & \text{otherwise};
  \end{cases} $$
 and for $p  \geq 0$ we have the following relation between their graded betti numbers.
$$
     \beta_{j,p}(S/(I+(\m))) =
    \begin{cases}
    \beta_{j,p}(S/I)+1      &    j=1, \  p=\deg(\m), \ \deg(\bn)\neq p \\
     \beta_{j,p}(S/I)+1     &    j=1, \  p=\deg(\m), \ \bm \nmid \bn\\
    \beta_{j,p}(S/I)-1      &    j=1, \  p=\deg(\bn), \  \bm \mid \bn\\
    \beta_{j,p}(S/I)+1      &    j=2, \  p=\deg(\lcm(\m,\bn)), \ \m\nmid \bn\\
    \beta_{j,p}(S/I)        & \text{otherwise;}
      \end{cases} 
$$
and finally we have the relation between their total betti numbers
$$\beta_{j}(S/(I+(\m))) =
  \begin{cases}
    \beta_{j}(S/I)+1         & j \in \{1,2\}, \quad \m\nmid \bn \\
    \beta_{j}(S/I)           & \text{otherwise.}
  \end{cases}
$$
\end{theorem}

\begin{proof}  Let
$\Delta$ be a simplicial complex supporting a free resolution of $I$ (e.g. the Taylor simplex).  According to
\cref{p:Bayer and Sturmfels}, if $\Gamma$ supports a free resolution of $S/(I+(\bm))$, then we can compute the multigraded betti
numbers of $S/(I+(\m))$ by the dimension of $
\tilde{H}_i(\Gamma_{ <\bu})$ for  $\bu\in \LL_{I+(\m)}$.
So our goal is to compare
$\rH_{i}(\Delta_{< \bu})$ with $\rH_{i}(\Gamma_{< \bu})$
for all $\bu \in \LL_{I+(\m)}$ and $i\geq -1$.

 First we assume that $\m\mid \bn$. By \cref{c:replacement} $\Gamma=\Delta$ supports a free resolution of $I+(\m)$. If we pick $\bu\in \LL_{I+(\m)}\setminus\{\m\}=\LL_{I}\setminus\{\bn\}$, then  $\Delta_{<\bu}=\Gamma_{<\bu}$ and so for all $i$ 
 $$\beta_{i,\bu}(S/I)=
 \rH_{i-2}(\Delta_{<\bu})=
 \rH_{i-2}(\Gamma_{<\bu})=
 \beta_{i,u}(S/(I+(\m)).$$ 
 What is left is the vertex $\m\in \LL_{I+(\bm)}$ and $\bn  \in \LL_I$.  
 Now $\Gamma_{<\bm}$ is empty and so $\beta_{1,\m}(S/I)=\dim\rH_{-1}(\Gamma_{<\bu})=1$. Moreover by \cref{P: m not in I}, $\bn\notin \LL_{I+(\bm)}$. Therefore  we have $$\beta_{i,\bu}(S/I+(\m))=\begin{cases}
			1 & i=1, \bu=\m\\
            0 & \bu=\bn \\
             \beta_{i,\bu}(S/I) & \text{otherwise.}
		 \end{cases}
$$

When $\m\nmid \bn$, by \cref{p:expansion}, we can find $\Gamma=\Delta \cup \{\m, \{\m,
\bn\}\}$ which supports a free resolution of $I+(\m)$. Similarly, we want to compare
$\rH_{i}(\Delta_{< \bu})$ with $\rH_{i}(\Gamma_{< \bu})$
for all $\bu \in \LL_{I+(\m)}$ and $i\geq -1$. There are four different scenarios.

\begin{enumerate}[(i)]
\item $\m \nmid \bu$ \label{en: 2}

In this case  $\lcm(\m, \bn) \nmid \bu$.  Therefore, $\Gamma_{< \bu}$ = $\Delta_{< \bu}$, and by \cref{p:Bayer and Sturmfels}, for all $j \geq 1$ 
$$\beta_{j,\bu}(S/(I+(\m))) = \beta_{j,\bu}(S/I).$$

\item $\m\mid \bu, \bu \notin \{\m,\lcm(\m,\bn)\}$ \label{en: 3}

We have $\m \mid \bu$, $\bu \neq \m$, and $\m \in C_I(\bn)$. So by \cref{P: C_I and "N_I"}, $\bn \mid \bu$. Therefore, 
$\lcm(\m,\bn) \mid \bu$, 
which implies, given that $\bu\neq \lcm(\m,\bn)$,
$$\{\m,\bn\} \in \Gamma_{< \bu}.$$
Hence,
$$\Gamma_{ <\bu} = \Delta_{ <\bu} \cup \{\m,\{\m,\bn\}\}.$$
By \cref{d:collapse}, given that $\bn \mid \bu$, and $\m$ is a free face of $\Gamma_{ <\bu}$, 
$$\Delta_{<\bu} = (\Gamma_{<\bu})_{\searrow {\{\m\}}},$$ or in other words $\Gamma_{ <\bu}$ collapses to $\Delta_{ <\bu}$ along the free face $\{\m\}$. Therefore, $$\rH_{j}(\Gamma_{< \bu}) = \rH_{j}(\Delta_{< \bu})$$ for all $\bu$ satisfied condition in \ref{en: 3} and  $j \geq -1$. By \cref{p:Bayer and Sturmfels}, for all $j \geq 1$ and $\bu$ as in \ref{en: 3} $$\beta_{j,\bu}(S/(I+(\m))) = \beta_{j,\bu}(S/I).$$    
    
\item $\bu = \lcm(\m,\bn)$ \label{en: 4}

Since $\m\in C_I(\bn)$, $\m\notin I$ by \cref{P: m not in I}. Then, $\Gamma_{ <\bu} = \Delta_{ <\bu} \cup \{\m\}$ by \cref{D:subcomplex}, which means that $\Gamma_{ <\bu}$ is $\Delta_\bu$ plus an isolated vertex labeled with $\m$. Since $\m\nmid \bn$, $\Delta_{ <\bu}$ contains at least the vertex $\bn$, and $\Gamma_{ <\bu} = \Delta_{ <\bu} \cup \{\m\}$. So, the dimension of the reduced homology group increases by one in homological degree 0.

We conclude that when $\bu=\lcm(\m,\bn)$, and using  \cref{p:Bayer and Sturmfels}, 
$$\dim\rH_{j}(\Gamma_{< \bu}) =
  \begin{cases}
    0     &   j=-1,\  \bm \mid \bn\\
    \dim \rH_{j}(\Delta_{< \bu}) + 1     &   j=0,\ \bm \nmid \bn\\
    \dim\rH_{j}(\Delta_{< \bu})             & \text{otherwise},\\
  \end{cases}
\Longrightarrow
 \beta_{j,\bu}(S/(I+(\m))) =
  \begin{cases}
    0         &   j=1, \   \m\mid \bn\\
   \beta_{j,\bu}(S/I)+1         &   j=2, \  \m\nmid \bn\\
   \beta_{j,\bu}(S/I)           & \text{otherwise}.
  \end{cases}
$$

\item $\bu = \m$ \label{en: 1}

In this case  $\Gamma_{<\bm}$ is empty since $\m \notin I$. Therefore, 
$$
 \dim\rH_{j}(\Gamma_{< \bm}) =
  \begin{cases}
    1     &   j=-1\\
    \dim\rH_{j}(\Delta_{< \bm}) & \text{otherwise}\\
  \end{cases}
\Longrightarrow
 \beta_{j,\bm}(S/(I+(\m))) =
  \begin{cases}
   1       &   j=1 \\
   \beta_{j,\bm}(S/I)           & \text{otherwise.}
  \end{cases}
$$
\end{enumerate}

Now we can summarize all four cases with $\m\nmid \bn$, along with the case $\m\mid \bn$. For  $\bu\in \LL_{I+(\m)}$,
\begin{align*}
 \beta_{j,\bu}(S/(I+(\m))) &=
    \left \{ \begin{array}{lll}
   1        &   j=1, \  \bu=\m & \mbox{(from\ \ref{en: 1})} \\
   0        &   j=1, \  \bu=\bn, \  \m \mid \bn & \mbox{(from\ \ref{en: 4})} \\
   \beta_{j,\bu}(S/I)+1        &   j=2, \  \bu=\lcm(\m,\bn), \  \m \nmid \bn\ &
                                  \mbox{(from\ \ref{en: 4})} \\
   \beta_{j,\bu}(S/I)           & \text{otherwise}. &
  \end{array} \right . 
\end{align*}

The equations for the graded and total betti numbers now follow from 
\eqref{E:Betti}.
\end{proof}

\begin{example}\label{R: 5}
Let the monomial ideal $I=(x_1^2x_2x_4,x_1x_2^2x_3^2,x_1x_2x_3^3)$. Using Macaulay 2 \cite{GS}, we obtain
$$\begin{array}{llll}
\beta_{1,x_1^2x_2x_4}(S/I)&= \beta_{1,x_1x_2^2x_3^2}(S/I) 
& =\beta_{1,x_1x_2x_3^3}(S/I) &= \beta_{2,x_1x_2^2x_3^3}(S/I) =\\
 \beta_{2,x_1^2x_2^2x_3^2x_4}(S/I)&=\beta_{2,x_1^2x_2x_3^3x_4}(S/I) 
& =\beta_{3,x_1^2x_2^2x_3^3x_4}(S/I) &= 1.
\end{array}$$
Now, we can use \cref{T:Betti number unchange} to directly obtain $\beta_{j,\bm}$, $\beta_{j,p}$ and $\beta_{j}$ of $S/(I + (\m_\gamma))$ where $\m_\gamma = x_2x_3^3x_4$. By \cref{T:Betti number unchange},
$$\begin{array}{lll}
\beta_{1, x_2x_3^3x_4}(S/(I + (\m_\gamma)) & =\beta_{2, x_1x_2x_3^3x_4}(S/(I + (\m_\gamma)) &= 1,\\
&&\\
\beta_{j, \bu}(S/(I + (\m_\gamma)) &= \beta_{j, \bu}(S/I) &\mbox{for all other } j\geq 0,\ \bu \in \LL_{I+(\m_\gamma)}, 
\end{array}$$
which is identical to the results from Macaulay 2 \cite{GS}.
\end{example}

%%%%%%%%%%%%%%%%%%%%%%%%%%%%%%%%%%
\section{Building classes of monomial ideals with similar betti numbers}\label{s:algorithm}
%%%%%%%%%%%%%%%%%%%%%%%%%%%%%%%%%%%%

Let $\m_1,\ldots,\m_q$ be monomials in $S$, and $I=(\m_1,\ldots,\m_q)$ with $\gcd(\m_1,\ldots,\m_q)\neq 1$. A natural question is: Can we apply \cref{T:Betti number unchange} to $I$ as many times as we wish? In other words, can we find as many as monomials $\m_{q+1},\m_{q+2},\ldots$ such that the betti numbers of $$I+(\m_{q+1})+(\m_{q+2})+\cdots$$ can be computed from the betti numbers of $I$ by  \cref{T:Betti number unchange}? 

The process of \cref{c:miw/v} is decreasing the degree of $\gcd(I)$ as
new generators are added to $I$, in the sense that the
$$\deg \big (\gcd(I+(\m))\big ) \lneq \deg \big (\gcd(I) \big )
  \qwhere \m\in C_I(\bn)
  \qand \bn \in \mingens(I),$$ so the number of times we can repeat
  \cref{c:miw/v} to grow the same ideal $I$ while controlling it betti
  numbers is bounded by the degree of $\gcd(I)$.

To overcome this limitation, we can add a shared variable to the
generators before applying \cref{T:Betti number unchange}. Then we can
use \cref{T:Betti number unchange} as many times as we desire. In
other words, the betti numbers
of $$\cdots(\bv_2(\bv_1I+(\m_{q+1}))+(\m_{q+2}))\cdots$$ can be
computed from the betti numbers of $I$ for monomials $$\bv_i \in S,
\quad \m_{q+1}\in C_{\bv_1 I}(\bv_1\bn), \quad \m_{q+2}\in
C_{\bv_2(\bv_1 I+(\m_{q+1}))}(\bw)$$ where $\bv_1\bn$ and $\bw$ are
generators of $\bv_1 I$ and $\bv_2(\bv_1 I+(\m_{q+1}))$,
respectively. Before stating this fact in detail, we give a comparison
of (multigraded) betti numbers of $\bv I$ and $I$.

\begin{lemma}\label{L:Betti numbers on multiplication} 
Let $I$ be a monomial ideal in $S$ and $\bu$, $\bv$ nontrivial monomials with $\bu \in \LL_{\bv  I}$. 
Then for all $i\geq 0$ $$\beta_{i,\bu}(S/\bv  I) = \beta_{i,\frac{\bu}{\bv} }(S/I), \quad \beta_{i}(S/\bv  I) = \beta_{i}(S/I).$$ 
\end{lemma}

\begin{proof}
	The second equation is a direct result of the first equation since the $i$-th betti number is the sum of all $i$-th multigraded betti numbers. So it suffices to prove that the first equation holds.  Now
	 \begin{align*}
		\LL_{\bv  I} =& \big \{\lcm(\bv \m_{1},\ldots ,\bv \m_{r}) \st \{\bv  \m_{1},\ldots ,\bv \m_{r}\} \subseteq \mingens(\bv I) \big \}\\
			=&\big \{\bv  \cdot \lcm(\m_{1},\ldots ,\m_{r}) \st \{ \m_{1},\ldots ,  \m_{r}\} \subseteq \mingens(I) \big \}\\
		=&\bv \LL_I.
	\end{align*}
In other words, there is a one to one correspondence, given by $\bu \longrightarrow \bv \bu$, between $\LL_I$ and $\LL_{\bv  I}$, which leads to an isomorphism of the lcm lattices. Now the order complex of the lcm lattice of a monomial ideal supports a free resolution of the ideal (\cite{GPW}), so if $\Delta$ is the order complex of  $\LL_I$ and $\Delta'$ is the order complex of $\LL_{\bv  I}$, for any $\bu\in \LL_{\bv I}$, 
$$\Delta_{<\frac{\bu}{\bv}} = \Delta'_{ <\bu}.$$ 
The equality $\beta_{i,\bu}(S/\bv  I) = \beta_{i,\frac{\bu}{\bv} }(S/I)$ now follows from \cref{p:Bayer and Sturmfels}.
\end{proof}

Combining \cref{L:Betti numbers on multiplication} and \cref{T:Betti number unchange} leads to the following statement.

\begin{theorem}[\bf{Expanding a monomial ideal}]\label{T:Betti number unchanged 2}
Let $I$ be a monomial ideal in the polynomial ring $S = \kappa[x_1, \ldots ,x_n]$. 
Supposed $\bv$ and $\bn$ are nontrivial monomials  in $S$, such that $\bv\neq 1$ and $\bn \in \mingens(I)$. Then $C_{\bv  I}(\bv\bn)\neq \emptyset$, and for every $\bm\in C_{\bv  I}(\bv\bn)$, we have

\begin{align*}
 \beta_{j,\bu}(S/(\bv I+(\m))) &=
  \begin{cases}
   1        &   j=1,\quad  \bu=\m \\
   0        &   j=1,\quad  \bu=\bv \bn, \quad \m \mid \bv \bn\\
   \beta_{j,\frac{\bu}{\bv}}(S/I)+1        &   j=2, \quad  \bu=\lcm(\m,\bv \bn),\quad  \m \nmid \bv \bn\ \\
   \beta_{j,\frac{\bu}{\bv}}(S/I)           & \text{otherwise};
  \end{cases} \\
  \beta_{j}(S/(\bv I+(\m))) &=
  \begin{cases}
    \beta_{j}(S/I)+1         &   j \in \{1,2\}, \quad  \m\nmid \bv \bn \\
    \beta_{j}(S/I)           & \text{otherwise.}
  \end{cases}
\end{align*}
for all $j \geq 0,$ $1\neq \bu\in \LL_{\bv I+(\m)}$
\end{theorem}

\begin{proof}
	By \cref{l:sizeofC_I}, $C_I(\bn) \neq \emptyset$ if $\gcd(I)\neq 1$ and $I\neq (x_1)$. So, $C_{\bv  I}(\bv\bn) \neq \emptyset$ for any monomial $\bv \neq 1$. Let $\m\in C_{\bv I}(\bv\bn)$. By \cref{L:Betti numbers on multiplication}, $\beta_{j,\bu}(S/\bv  I) = \beta_{j,\frac{\bu}{\bv} }(S/I)$ for all $1\neq \bu\in \LL_{\bv I}$. The statement now follows directly from \cref{T:Betti number unchange}.
\end{proof}

\begin{remark}\label{R: convenientbuild}
	There is a convenient way to create a new monomial ideal with same betti numbers by \cref{T:Betti number unchanged 2}. For any $I=(\m_1,\ldots,\m_q)$, choose monomial $\bv$ such that  $$\gcd(\bv, \bm_i)=1 \quad \forall i\in [q].$$ 
	Then, by \cref{c:miw/v}, $$\bm_i=\frac{\bv \bm_i\cdot1}{\bv}\in C_{\bv I}(\bv\bm_i) \quad \forall i \in [q].$$ According to \cref{T:Betti number unchanged 2}, for all $j \geq 0$ $$\beta_j(I)=\beta_j(\bv I+\bm_i) \quad \forall i \in [q],$$ which is same as saying $$\beta_j(\m_1,\ldots,\m_q)= \beta_j(\bv \m_1,\ldots,\bv \m_{i-1},\, \m_i,\, \bv \m_{i+1},\ldots,\bv \m_q).$$
	For example, 
	$$\beta_j(x_1^2x_2,x_2x_3,x_1x_4)= \beta_j(x_1^2x_2x_5, x_2x_3,x_1x_4x_5).$$
	
\end{remark}
We can repeat \cref{T:Betti number unchanged 2} as many times as possible from any monomial ideal $I$ in $S = \kappa[x_1,\ldots,x_n], n\geq 2$. 

\begin{corollary}[\bf{Multiple expansions of a monomial ideal}]\label{C: infinitely expand ideal} 
    Let $I$ be an ideal minimally generated by monomials $\m_1,\ldots, \m_q$ in the polynomial ring $S=\kappa[x_1,\ldots,x_n]$, $n\geq 2$, $\kappa$ a field, and let $\bv_1,\ldots,\bv_s$ be any nontrivial monomials in $S$. Construct the ideals $I_1,\ldots,I_s$ where 
    $$    I_1=\bv_1I+(\bu_1),$$ and for $j>1$  
    $$I_j=\bv_1\cdots \bv_j I + 
    (\bv_2\cdots \bv_{j}\bu_1)+
      (\bv_3\cdots \bv_{j}\bu_2)+
    \cdots+(\bv_j\bu_{j-1})+(\bu_j),$$ 
    where  $\bu_1,\ldots,\bu_s$ are monomials such that $$\bu_j\in C_{\bv_jI_{j-1}}(\bw_j)  \qfor j\in [s], \qand \bw_{j-1} \in  \mingens(\bv_jI_{j-1}).$$
    If, for every $j\in [s]$, $\bv_j\cdots \bv_{s-j}\bu_j$ does not divide $\bw_{j-1}$,
     then 
    \begin{align*}
      \beta_{j}(S/I_s) &=
 \begin{cases}
    \beta_{j}(S/I)+ s         &   j \in \{1,2\} ,\\
    \beta_{j}(S/I)           & \text{otherwise.}
  \end{cases}
  \end{align*}

\end{corollary}

%%%%%%%%%%%%%%%%%%%%%%%%%%%%%%%%%%%%%%%%%%%%%%%%%%%%%%%%%%%%
\section{Applications to projective dimension, regularity} \label{s:last section}
%%%%%%%%%%%%%%%%%%%%%%%%%%%%%%%%%%%%%%%%%%%%%%%%%%%%%%%%%%%%

It is natural to wonder how do the homological invariants change while expanding $I$. This is the topic of the current section. Suppose $\bv$ is a monomial in $S$ so that $\gcd(\bv I)\neq 1$. Then by \cref{T:Betti number unchanged 2}, the $i$-th (multigraded) betti numbers of $I$ and $\bv I + \m$ are always same when $i\geq 2$,   $\m\in C_{\bv I}(\bn)$ and $\bn \in \mingens(\bv I)$. Therefore, the projective dimension of $I$ and $\bv I + (\bm)$ are the same if $\pd(S/I)\geq 2$.

\begin{theorem}[\bf{Projective dimension, depth and regularity of expansion}]\label{t:pd-d-r}
	Let $I$ be a monomial ideal in $S=\kappa[x_1,\ldots,x_n]$, $\bv$, $\bn$, and $\bm$ monomials such that $\gcd(\bv I) \neq 1$, $\bn \in \mingens (\bv I)$ and  $\m\in C_{\bv I}(\bn)$. Then
 \begin{enumerate}
 \item \label{i:pd} (Projective dimension)
 
\begin{equation*}
	\pd(S/(\bv I+(\m)))= \begin{cases} 2 &   \pd(S/I) = 1 \qand \m \nmid \bn, \\                                      
	\pd(S/I) & \mbox{otherwise}.
	\end{cases}
\end{equation*}

\item \label{i:depth} (Depth)  
\begin{equation*}
	\depth(S/(\bv I+(\m)))= 
 \begin{cases} n-2 &   \depth(S/I) = n-1 \qand \m \nmid \bn, \\ 
 \depth(S/I) & \mbox{otherwise}.
	\end{cases}
\end{equation*}

\item \label{i:reg} (Regularity) we can either increase the regularity arbitrarily large or keep it constant when $\bv = 1$. This depends on the $\m$ we choose. 

	\begin{align*}
		\reg(S/(I+(\m)))=\begin{cases}
			\reg(S/I) &   \deg(\m)-1 \leq \reg(I)\text{ and } \\ &\deg(\lcm(\m,\bn))-2\leq \reg(I)\\
			\max\{\deg(\m)-1,\deg(\lcm(\m,\bn))-2\} & \text{otherwise.}
		\end{cases}
	\end{align*}
\end{enumerate}
\end{theorem}

\begin{proof}
	Since the projective dimension is the maximum of the homological degree with non-zero betti numbers, it will not change if the betti numbers change in lower (or equal) homological degree. By \cref{T:Betti number unchanged 2}, 
    $\beta_i(\bv I + (\m))=\beta_i(I)$ for all $i>2$. Hence, the projective dimension of two monomial ideals are the same if the projective dimension of $I$ is greater or equal to 2. 
	
	Suppose $\pd(S/I) = 1$. By \cref{T:Betti number unchanged 2}, $\beta_2(S/ (\bv I + (\m)))=\beta_2(S/I)=0$ if $\m$ divides a generator of $\bv I$. Otherwise, $\beta_2(S/(\bv I +(\m)))=\beta_2(S/I)+1=1$.  Hence, $\pd(S/(\bv I+(\m))) = \pd(S/I)$ and $\pd(S/(\bv I+(\m))) = 2$ respectively. This settles \eqref{i:pd}.

    By the Auslander-Buchsbaum theorem, $\depth(S/(\bv I + (\m)))=n-\pd(S/(\bv I + (\m)))$, and so \eqref{i:depth} follows.
 
	To see \eqref{i:reg}, note that by \cref{T:Betti number unchange}, $\beta_{i,j}(S/(I+(\m)))=\beta_{i,j}(S/I)$ except when  $$i=1,j=\deg(\m) \qor  i=2,j=\deg(\lcm(\m,\bn)).$$ Therefore the regularity will change (increase) if and only if $$\deg(\m) > \reg(S/I) +1 \qor \deg(\lcm(\m,\bn)) >\reg(S/I) +2.$$
\end{proof}

The following example is the application of \cref{t:pd-d-r}\eqref{i:reg}.

\begin{example}\label{ex:arbitrary-reg}
	Let $I=(abd,a^2b^2,ac^3e,a^2c^2)\subset \kappa[a,b,c,d]$ . Then $\gcd(I)=a$. Using Macaulay 2 \cite{GS},  we can compute that $\reg(S/I) = 5$. 
    By \cref{c:miw/v}, $\m\in C_{I}(abd)$ is in the form of $b^{1+y_2}c^{y_3}d^{1+y_4}$ 
	where $y_2,y_3,y_4\geq 0$. 
    
    According to \cref{t:pd-d-r}\eqref{i:reg}, if we want $\reg(I+(\m)) = \reg(I)$, then  $y_2+y_3+y_4\leq 4$. For instance, $\m=b^2cd^2$. 
    
    If we wish to have  $\reg(I+(\m))=1000$, then we can let let $y_2+y_3+y_4= 999$. For instance, $\m= b^{500}cd^{500}$. 
\end{example}

%%%%%%%%%%%%%%%%%%%%%%%%%%%%%%%%%%%%%%%%%%%%%%%
\section{Polarization and expansion: reduction to square-free monomial ideals}\label{s:pol}
%%%%%%%%%%%%%%%%%%%%%%%%%%%%%%%%%%%%%%%%%%%%%%%
Polarization - a tool which transforms a monomial ideal with a square-free one - will allow us to use Stanley-Reisner theory to study
the Cohen-Macaulay properties of expansion ideals via Stanley-Reisner theory. 

The {\bf polarization} of a monomial $\bm=x_1^{a_1}\cdots x_n^{a_n}$ in the polynomial ring $S=\kappa [x_1,\ldots,x_n]$
is a square-free monomial 
$$\P(\bm)=x_{1,1} x_{1,2} \cdots 
x_{1,a_1} x_{2,1}\cdots x_{2,a_2}\cdots 
x_{n,1}\cdots x_{n,a_n}
\quad \mbox{in the polynomial ring} \quad 
S_{\P(\bm)}=\kappa [x_{i,j} \st 1 \leq i \leq n, 1 \leq j \leq a_i].$$

If $I$ is an ideal minimally generated by monomials $\bm_1,\ldots,\bm_q$ where $\lcm(\bm_1,\ldots,\bm_q)=x_1^{b_1}\cdots x_n^{b_n}$, then we define the {\bf polarization of $I$} to be the square-free monomial ideal 
$$\P(I)=(\P(\bm_1),\ldots,\P(\bm_q)) 
\quad \mbox{in the polynomial ring} \quad 
S_{\P(I)}=\kappa [x_{i,j} \st 1 \leq i \leq n, 1 \leq j \leq b_i].$$

The operation of polarization is done via a quotient of a regular sequence, 
and preserves many properties of the ideal, such as its betti numbers, the (sequential) Cohen-Macaulay property, and more.

\begin{theorem}[{\bf Polarization and Expansion}]\label{t:polar} Let $I$ be a monomial ideal in $S=\kappa[x_1,\ldots,x_n]$  and let $\bn$ and $\bm$ be monomials such that $\bn \in \mingens (I)$. Then   
$$\bm\in C_{I}(\bn) \quad \mbox{ if and only if } \quad 
\P(\bm) \in C_{\P(I)}(\P(\bn)).$$
\end{theorem}

\begin{proof} 
     Suppose $\bm\in C_I(\bn)$. By \cref{c:miw/v}, $\m=\frac{\bn \bw}{\bv}$ where 
   
   \begin{enumerate}
 \item $\bv \neq 1$;
 \item if $\bw =1 $ then  $\bv\neq \bn$;
 \item $\gcd \Big(\bv, \frac{\bn}{\gcd(I)} \Big)=1$ (in particular $\bv \mid \gcd(I)$);
 \item $\gcd(\bv,\bw) = 1$ (in particular $\bv \mid \bn$).
 \end{enumerate}

 By (3), $\bv\mid \gcd(I)\mid \bn$. Without loss of generality, we permute the variables $x_1,..,x_n$ so that we can assume for positive integers $0<h \leq i \leq j$  
 $$
 \bv=x_1^{v_1}\cdots x_i^{v_i},  \quad 
 \bn=x_1^{\eta_1}\cdots x_j^{\eta_j}, \qand
 \bw=x_{i+1}^{w_{i+1}}\cdots x_{n}^{w_{n}}
 $$ 
 where 
 $$ 
 0\leq w_t, \quad
 v_a = \eta_a, \quad 
 0< v_b <\eta_b, \qand
 0< \eta_c \qfor 
 i < t \leq n, \quad 
 1 \leq a < h, \quad 
 h \leq b \leq i, \qand
 i <c \leq j.
 $$
Then 
$$\bm=\frac{\bn\bw}{\bv}=
x_h^{\eta_h-v_h}\cdots x_i^{\eta_i-v_i} \cdot
x_{i+1}^{\eta_{i+1}+w_{i+1}}\cdots x_{j}^{\eta_{j}+w_j} \cdot
x_{j+1}^{w_{j+1}}\cdots x_{n}^{w_n}
$$
and 
$\displaystyle \P(\bm)= \frac{\P(\bn)\tilde{\bw}}{\tilde{\bv}}$
where 
$$ \begin{array}{ll}
 \tilde{\bv}=&\P(x_1^{\eta_1}\cdots x_{h-1}^{\eta_{h-1}})\cdot
 x_{h,\eta_h-v_h+1}\cdots x_{h,\eta_h}\cdots x_{i,\eta_i-v_i+1}\cdots x_{i,\eta_i}\\
 \tilde{\bw}= &x_{i+1,\eta_{i+1}+1}\cdots x_{i+1,\eta_{i+1}+w_{i+1}} \cdots 
 x_{j,\eta_{j}+1}\cdots x_{j,\eta_{j}+w_{j}} \cdot 
 \P(x_{j+1}^{w_{j+1}}\cdots x_{n}^{w_n}).
 \end{array}$$
It remains to show the four conditions in \cref{c:miw/v} are satisfied, and this follows directly from the construction above, the fact that $\gcd(\P(I))=\P(\gcd(I))$, and the observation that $\bv$ (respectively $\bw$) is a nontrivial monomial if and only if $\tilde{\bv}$ (respectively $\tilde{\bw}$) is a nontrivial monomial.

The other direction ($\Longleftarrow$) is similar: if 
$\P(\bm) \in C_{\P(I)}(\P(\bn))$, we can write 
$$\P(\bm)=\frac{\P(\bn)\bw'}{\bv'}$$ for some 
$\bw', \bv' \in S_{\P(I)}$ such that
 \begin{enumerate}
 \item $\bv' \neq 1$;
 \item if $\bw' =1 $ then  $\bv'\neq \P(\bn)$;
 \item $\gcd \Big(\bv', \frac{\P(\bn)}{\gcd(\P(I))} \Big)=1$ (in particular $\bv' \mid \gcd(\P(I))$);
 \item $\gcd(\bv',\bw') = 1$ (in particular $\bv' \mid \P(\bn)$).
 \end{enumerate}

Suppose 
$$\bm=x_1^{m_1}\cdots x_n^{m_n} \qand 
\bn=x_1^{\eta_1}\cdots x_n^{\eta_n}$$ 
and set 
$$
\bv=x_1^{v_1}\cdots x_n^{v_n} \qand  
\bw=x_1^{w_1}\cdots x_n^{w_n}
$$
where for each $i \in [n]$ let
$$v_i=
\begin{cases} \min(j \st x_{i,j} \mid \bv')-1 & \mbox{if } x_{i,j} \mid \bv' \mbox{ for some } j\\
                   \eta_i & \mbox{otherwise,}
\end{cases}
\qand 
w_i=\begin{cases} \max(j \st x_{i,j} \mid \bw') & \mbox{if } x_{i,j} \mid \bw' \mbox{ for some } j\\
                   0 & \mbox{otherwise.}
\end{cases}
$$

Since $\bv' \mid \P(\bn)$, if $x_{i,j} \mid \bv'$, then $x_{i,j} \mid \P(\bn)$, and so $v_i \leq \eta_i$. In particular $\bv \mid \bn$. For a similar reason $\bw \mid \bm$.

As before,  we permute the variables $x_1,..,x_n$ so that we can assume  
 $$
 \bv=x_1^{v_1}\cdots x_n^{v_n}, \qand \bn=x_1^{\eta_1}\cdots x_j^{\eta_j}, \quad
 \bw=x_{i+1}^{w_{i+1}}\cdots x_{n}^{w_{n}}
 $$ 
 where 
  \begin{align*} 
 0<v_u \lneq \eta_u &\qfor    u=1, \ldots, i\\
 0\leq w_u \qand  v_u=\eta_u &\qfor u = i+1,\ldots,n.
\end{align*}

It follows that $\frac{\P(\bn)}{\P(\bv)}=\bv'$, and 
 \begin{align*} 
 m_u =v_u  &\qfor    u=1, \ldots, i \\
 m_u = \eta_u + w_u & \qfor      u= i+1,\ldots,n .
 \end{align*}

Now let 
$\tilde{\bv}= \frac{\bn}{\bv}
= x_{1}^{\eta_1-v_1}\cdots x_i^{\eta_i-v_i}$. 
Then
$$
\bm=\frac{\bn\bw}{\tilde{\bv}}= 
x_1^{v_1}\cdots x_i^{v_i}
\cdot x_{i+1}^{\eta_{i+1}+w_i}\cdots  x_{n}^{\eta_n+w_{n}}
$$
and we observe that
\begin{enumerate} 
\item $\tilde{\bv}\neq 1$ since $\bv' \neq 1$;

\item If $\bw=1$ then $\bw'=1$, and so $\P(\bm)=\frac{\P(\bn)}{\bv'}$. If also $\tilde{\bv}= \bn$ then  $v_1=\cdots=v_j=1$, which means that $x_{1,1}\cdots x_{j,1} \mid \bv'$, implying that for every $x_u \mid \bm$, $x_{u,1} \nmid \P(\bm)$, which is impossible unless $\bm=1$, or equivalently $\bv'=\P(\bn)$, which contradicts our assumption;

\item 
$\gcd\big(\tilde{\bv},\frac{\bn}{\gcd(I)}\big)=\gcd \big (\frac{\bn}{\bv},\frac{\bn}{\gcd(I)} \big )=\gcd \big (\frac{\P(\bn)}{\P(\bv)}, \frac{\P(\bn)}{\P(\gcd(I))} \big ) 
 =\gcd \big (\bv', \frac{\P(\bn)}{\gcd(\P(I))} \big)=1$;

\item $\tilde{\bv}$  is coprime to $\bw$ by their construction.
\end{enumerate}
So all four conditions in \cref{c:miw/v} are satisfied. We are done.

\end{proof}

%%%%%%%%%%%%%%%%%%%%%%%%%%%%%%%%%%%%%%%%%%%%%%%
\section{The Cohen-Macaulay property}\label{s:SCM}
%%%%%%%%%%%%%%%%%%%%%%%%%%%%%%%%%%%%%%%%%%%%%%%
 Since the operation of adding generators to an  ideal $I$ requires $\gcd(I)=1$, most of the ideals of concern to us are not unmixed, and hence not Cohen-Macaulay. But we can show that if each unmixed component of $I$ is Cohen-Macaulay to begin with, then the expansion ideal will also have Cohen-Macaulay unmixed components. In other words the expansion of a sequentially Cohen-Macaulay ideal remains sequentially Cohen-Macaulay. Thanks to \cref{t:polar}, the problem can be reduced to the square-free case. We  first introduce the one-to-one correspondence between monomials $\bu$ in the polynomial ring $S=\kappa[x_1,\ldots,x_n]$
and subsets $A$ of $[n]$:
$$ \tau_\bu=\{i \in[n] \st x_i \mid \bu\}  \quad \qand \quad 
\bm_A=\prod_{i \in A} x_i.$$

Let $I$ be a square-free monomial ideal in $S$. We associate to $I$ a (unique) simplicial complex $\Gamma$ on the vertex set $[n]$ where $$\Gamma=\{A \subset [n] \st \bm_A \notin I\}.$$

The simplicial complex $\Gamma$ is called the {\bf Stanley-Reisner complex} of $I$,  and $I$ is called
the {\bf Stanley-Reisner ideal} of $\Gamma$.

\begin{example}\label{ex:SR} The Stanley-Reisner complexes for $I=(abc,abd)$ and $I+(\bm)=(abc,abd,cd)$ where $\bm=cd \in C_I(abc)$ appear in \cref{f:SR}. 

\begin{figure}%[th!]
$$\begin{array}{cc}
\begin{tikzpicture}
\coordinate (A) at (-1,2);
\coordinate (C) at (0, 0);
\coordinate (B) at (1, 2);
\coordinate (D) at (0, 1 );
\draw[draw = none, fill=lightgray] (A) -- (C) -- (D) ;
\draw[draw = none, fill=lightgray] (B) -- (C) -- (D) ;
\draw[black, fill=black] (A) circle(0.04);
\draw[black, fill=black] (B) circle(0.04);
\draw[black, fill=black] (C) circle(0.04);
\draw[black, fill=black] (D) circle(0.04);
\draw[-] (A) -- (B);
\draw[-] (A) -- (C);
\draw[-] (A) -- (D);
\draw[-] (B) -- (C);
\draw[-] (B) -- (D);
\draw[-] (C) -- (D);
\node[label = below :$a$] at (A) {};
\node[label = below :$b$] at (B) {};
\node[label = below :$c$] at (C) {};
\node[label = above:$d$] at (D) {};
\end{tikzpicture}
\qquad & \qquad 
\begin{tikzpicture}
\coordinate (A) at (-1,2);
\coordinate (C) at (0, 0);
\coordinate (B) at (1, 2);
\coordinate (D) at (0, 1 );
\draw[black, fill=black] (A) circle(0.04);
\draw[black, fill=black] (B) circle(0.04);
\draw[black, fill=black] (C) circle(0.04);
\draw[black, fill=black] (D) circle(0.04);
\draw[-] (A) -- (B);
\draw[-] (A) -- (C);
\draw[-] (A) -- (D);
\draw[-] (B) -- (C);
\draw[-] (B) -- (D);
\node[label = below :$a$] at (A) {};
\node[label = below :$b$] at (B) {};
\node[label = below :$c$] at (C) {};
\node[label = above:$d$] at (D) {};
\end{tikzpicture}
\end{array}
$$
\caption{Stanley Reisner complexes for $I$ (left) and $I+(\bm)$ (right) in \cref{ex:SR}}\label{f:SR}
\end{figure}
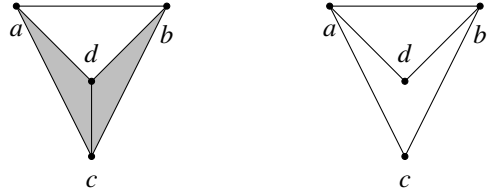

\end{example}

If $I$ is the  Stanley-Reisner ideal of $\Gamma$,  
a theorem of Duval~\cite{Duval} states that $I$   is sequentially Cohen-Macaulay if and only if for every $i\in \{0,\ldots, \dim(\Gamma)\}$, the Stanley-Reisner ideal of its pure $i$-skeleton $\Gamma^{[i]}$ -- the subcomplex whose facets are the $i$-dimensional faces of $\Gamma$ -- is Cohen-Macaulay. Putting this together with  Reisner's criterion~\cite{Reisner} for Cohen-Macaulayness, we conclude that $I$ is sequentially Cohen-Macaulay if and only if 
\begin{equation}\label{e:Reisner}
\tilde{H_j}(\lk_{\Gamma^{[i]}}(\sigma),\kappa) =0 
\qforall
1 \leq i \leq  \dim(\Gamma), \quad 
\sigma \in \Gamma^{[i]} \qand
j < \dim(\lk_{\Gamma^{[i]}})=i- |\sigma|.
\end{equation}

\begin{theorem}[{\bf Expansion preserves sequential Cohen-Macaulay properties}]\label{t:scm} Let $I$ be a square-free monomial ideal in $S=\kappa[x_1,\ldots,x_n]$ and let $\bn$ and $\bm$ be square-free monomials such that $\bn \in \mingens (I)$ and  $\m\in C_{I}(\bn)$. If $I$ is sequentially Cohen-Macaulay, then so is $I+(\bm)$.
\end{theorem}

\begin{proof} 
Suppose $\Gamma$ and $\Omega$ are the Stanley-Reisner complexes of $I$ and $I+(\bm)$, respectively. 
Our goal is to show \eqref{e:Reisner} for $\Omega$, assuming it holds for $\Gamma$.

Since $\bm \notin I$ (\cref{P: m not in I}), we have $\tau_\bm \in \Gamma$. Applying \cref{c:miw/v}, we know there are (square-free) monomials $\bv$ and $\bw$ such that $\bv \mid \gcd(I)$, $\gcd(\bv,\bw)=1$ and 
$$\bm=\frac{\bn\bw}{\bv}.$$ 
In particular 
\begin{equation}\label{e:tauvwmn}
\tau_\bv \subseteq \tau_\bn, 
\quad 
\tau_\bv \cap \tau_\bw=\emptyset 
\qand 
\tau_\bm=(\tau_\bn\setminus \tau_\bv) \cup \tau_\bw.
\end{equation}

\noindent \emph{Claim.} With notation as in \eqref{e:tauvwmn},  
we claim that 
\begin{equation}\label{e:new-cup-bare}
\Gamma=
\Omega \cup \langle [n]\setminus \{x\} \st x \in \tau_\bv\rangle.
\end{equation}

One direction of the equality is clear: $I \subseteq I+(\bm)$ results in
$\Omega \subseteq \Gamma$. Moreover, for each face 
$\sigma_x=[n]\setminus \{x\}$
where 
$x \in \tau_\bv$, 
since 
$\bv \mid \gcd(I)$, 
we can conclude that 
$\gcd(I) \nmid \bm_{\sigma_x}$
 and hence
$\bm_{\sigma_x} \notin I$. Therefore $\sigma_x \in \Gamma$ and  
we have established that
$$
\Gamma \supseteq 
\Omega \cup  \langle [n]\setminus \{x\} \st x \in \tau_\bv\rangle. 
$$

For the reverse inclusion, let $\sigma$ be a face of $\Gamma$, so that $\bm_\sigma \notin I$. Then either $\bm_\sigma \notin I+(\bm)$, in which case $\sigma \in \Omega$, or $\bm \mid \bm_\sigma$. If the latter happens, since $\bm_\sigma \notin I$, we have 
$\bn \nmid \bm_\sigma$ and $\bm \mid \bm_\sigma$. In other words  using also \eqref{e:tauvwmn} 
$$\tau_\bn \not \subseteq \sigma  
\qand  
\tau_\bm = (\tau_\bn\setminus \tau_\bv) \cup \tau_\bw \subseteq \sigma
$$
which gives us  
$\sigma=\tau_\bm \cup \gamma$  
where 
$\tau_\bv \not \subseteq  \gamma$, 
or in other words, 
$$
\sigma \in \langle [n]\setminus \{x\} \st x \in \tau_\bv\rangle. 
$$
This settles \eqref{e:new-cup-bare}.

Set $\overline{\tau_\bv}=[n]\setminus \tau_\bv$  
and 
\begin{equation}\label{e:Sigma}
\partial \langle \tau_\bv \rangle = 
\langle \tau_\bv \setminus \{x\} \st x \in \tau_\bv \rangle
\qand 
\Sigma=\langle \overline{\tau_\bv}\setminus \tau_\bm\rangle * \partial \langle \tau_\bv \rangle.
\end{equation}

Now observe that the only faces of $\Gamma$ which are not in $\Omega$ are those containing $\tau_\bm$, and so for $i \geq 0$, if we let $\Gamma^{[i]}$ denote the pure-$i$-skeleton
of  $\Gamma$, 
 then
\begin{equation}\label{e:G-i}
\Gamma^{[i]}=
\begin{cases}
\Omega^{[i]} & \qif i< \deg(\bm)-1\\
\Omega^{[i]} \cup (\langle\tau_\bm \rangle * \Sigma^{[d]}) 
& \quad \mbox{otherwise},
\end{cases}
\end{equation}
where
$d=i+1-\deg(\bm)$.

We are now ready to establish \eqref{e:Reisner} for $\Omega$. 
Let $\sigma \in \Omega^{[i]}$ for some $i \in \{1,\ldots, \dim(\Omega)\}$.

If $i< \deg(\bm)-1$ or $\tau_\bm \cup \sigma \notin \Gamma$, 
then $\lk_{\Gamma^{[i]}}(\sigma)=
\lk_{\Omega^{[i]}}(\sigma)$. Then it follows immediately that 
if $j<\dim(\lk_{\Omega^{[i]}}(\sigma))$, 
$$\rH_j(\lk_{\Omega^{[i]}}(\sigma))= \rH_j(\lk_{\Gamma^{[i]}}(\sigma))=0 $$ 
as desired, and we are done.

If $i\geq \deg(\bm)-1$ and  $\tau_\bm \cup \sigma \in \Gamma$, we write 
$$\sigma=\sigma_1 \cup \sigma_2 
\qwhere  
\sigma_1=\sigma \cap \tau_\bm 
\qand 
\sigma_2=\sigma \setminus \sigma_1.$$

Then $\tau_\bm \cup \sigma =\tau_\bm \cup \sigma_2 \in (\langle\tau_\bm \rangle * \Sigma)^{[i]}$, and so by \eqref{e:G-i}, $\tau_\bm \cup \sigma_2$  a facet containing it is of the form
$$\tau_\bm \cup A \qwhere 
A \cap \tau_\bm=\emptyset, \quad
\tau_\bv \not \subseteq A, \qand 
|A|=d.$$
Then 
$$(\tau_\bm \cup A)  \setminus \sigma =
(\tau_\bm \setminus \sigma_1) \cup 
(A  \setminus \sigma_2)
\in \lk_{\Gamma^{[i]}}(\sigma).$$
Since $\sigma \in \Omega$, 
$\tau_\bm \not \subseteq \sigma$, and hence 
$\tau_\bm \setminus \sigma_1 \neq \emptyset$. Therefore
$$(\tau_\bm \cup A)  \setminus \sigma \in
\langle \tau_\bm \setminus \sigma_1 \rangle *\lk_{\Sigma^{[d]}}(\sigma_2).$$
We have shown that 
\begin{equation}\label{e:new-cup}
\lk_{\Gamma^{[i]}}(\sigma)=
\begin{cases}
\lk_{\Omega^{[i]}}(\sigma) & \qif i< \deg(\bm)-1 \qor \tau_\bm \cup \sigma \notin \Gamma \\
\lk_{\Omega^{[i]}}(\sigma) 
\cup
\left ( \langle \tau_\bm\setminus \sigma_1 \rangle * \lk_{\Sigma^{[d]}}(\sigma_2 )
\right )
& \quad \mbox{otherwise}.
\end{cases}
\end{equation}

By \eqref{e:new-cup}, the only case we need to discuss is when  
$\tau_\bm \cup \sigma \in \Gamma$ and $i \geq \deg(\bm)$, for which  we set up a Mayer-Vietoris exact sequence
based on \eqref{e:new-cup} below. 
\begin{equation}\label{e:MV-new}
\begin{array}{rll}
\cdots \to&
\rH_{j+1}(\lk_{\Gamma^{[i]}}(\sigma))
&\to
\rH_{j}\left (\lk_{\Omega^{[i]}}(\sigma) \cap 
\big ( \langle \tau_\bm\setminus \sigma_1 \rangle * \lk_{\Sigma^{[d]}}(\sigma_2)\big ) \right )
\to
\rH_{j}(\lk_{\Omega^{[i]}}(\sigma)) 
\oplus
\rH_{j}\left (\langle \tau_\bm\setminus \sigma_1  \rangle * \lk_{\Sigma^{[d]}}(\sigma_2) \right )\\
&& \\
\to  &\rH_{j}(\lk_{\Gamma^{[i]}}(\sigma))
    &\to \cdots 
\end{array}
\end{equation}

We focus on each component of \eqref{e:MV-new}.
\begin{enumerate}
    \item \label{i:cup} $\rH_{j}(\lk_{\Gamma^{[i]}}(\sigma))=0$ for all $j<i-|\sigma|$ by assumption.

\item\label{i:sum}  Since $\tau_\bm \setminus \sigma_1 \neq \emptyset$, $\rH_{j}\left (\langle \tau_\bm\setminus \sigma_1  \rangle * \lk_{\Sigma^{[d]}}(\sigma_2) \right )=0$ for all $j$.
\item \label{i:cap} The intersection $\lk_{\Omega^{[i]}}(\sigma) \cap (\langle \tau_\bm\setminus \sigma_1  \rangle * \lk_{\Sigma^{[d]}}(\sigma_2))$ contains faces $\gamma$ with
$\gamma \cap \sigma =\emptyset$ and where 
$$\gamma=\gamma_1 \cup \gamma_2, \quad
\gamma_1=\gamma \cap \tau_\bm, \quad  
\tau_\bm \supsetneq \gamma_1 \cup \sigma_1, \quad 
\sigma_2 \cup \gamma_2 \in 
\Sigma^{[d]},$$
which means 
$$\gamma_1 \in \lk_{\partial \langle \tau_\bm \rangle}(\sigma_1) 
\qand
\gamma_2 \in 
\lk_{\Sigma^{[d]}}( \sigma_2 ).$$ Therefore 
\begin{equation}\label{e:SO-cap-link-new}
\lk_{\Omega^{[i]}}(\sigma) 
\cap 
(\langle(\tau_\bm\setminus \sigma_1) \rangle * \lk_{\Sigma^{[d]}}(\sigma_2))=
 \lk_{\partial \langle\tau_\bm\rangle}(\sigma_1) 
 *
 \lk_{\Sigma^{[d]}}(\sigma_2).\end{equation}   

From \eqref{e:SO-cap-link-new} (see also~\cite{Bjorner}) we now have 
\begin{equation}\label{e:H*}
\rH_j(\lk_{\Omega^{[i]}}(\sigma) \cap 
(\langle(\tau_\bm\setminus \sigma_1) \rangle * \lk_{\Sigma^{[d]}}(\sigma_2)))
\cong 
 \bigoplus_{a+b=j-1} 
\rH_a(\lk_{\partial \langle\tau_\bm\rangle}(\sigma_1))
\otimes 
\rH_b(\lk_{\Sigma^{[d]}}(\sigma_2)).
\end{equation}

If $\omega=\tau_\bm \setminus \sigma_1$, then the facets of  $\partial \langle\tau_\bm \rangle$ containing $\sigma_1$ are $\tau_\bm \setminus \{w\}$ where $w \in \omega$. Therefore
$$ \lk_{\partial \langle\tau_\bm\rangle}(\sigma_1)=\partial \langle \omega \rangle,$$ and so 
$$\rH_a(\lk_{\partial \langle\tau_\bm \rangle}(\sigma_1))=0 \qforall a\neq |\omega|-2=\deg(\bm)-|\sigma_1|-2.$$
Substituting  in \eqref{e:H*}, we have 
\begin{equation}\label{e:H*1}
\rH_j(\lk_{\Omega^{[i]}}(\sigma) \cap 
\lk_{\langle \tau_\bm \rangle * \Sigma^{[d]}}(\sigma))
\cong 
\rH_{\deg(\bm)-|\sigma_1|-2}(\lk_{\partial \langle\tau_\bm \rangle}(\sigma_1))
\otimes 
\rH_{j-\deg(\bm)+|\sigma_1|+1}(\lk_{\Sigma^{[d]}}(\sigma_2)).
\end{equation}

If $j<i-|\sigma|$, recalling that   $d=i+1-\deg(\bm)$ and $|\sigma|=|\sigma_1|+|\sigma_2|$, we have  
\begin{equation}\label{e:j-deg}
j-\deg(\bm)+|\sigma_1|+1 
< i -|\sigma|  -\deg(\bm)+|\sigma_1|+1=
  d -|\sigma_2|=
  \dim(\lk_{\Sigma^{[d]}}(\sigma_2)).
\end{equation}

Using the description of $\Sigma$ as in
\eqref{e:Sigma}, and keeping in mind that $\sigma_2 \cap \tau_\bm=\emptyset$, let
$$\sigma_2 '=\sigma_2 \cap \overline{\tau_\bv}, \quad  \sigma_2 ''=\sigma_2 \cap \tau_\bv.$$
Since $\sigma_2'' \in \partial \langle\tau_\bv \rangle$, we must have $\sigma_2'' \neq \tau_\bv$.

If we let $\zeta=\tau_\bv 
\setminus \sigma_2''$, then the facets of  $\Sigma^{[d]}$ containing $\sigma_2$ are $(\overline{\tau_\bv}\setminus \tau_\bm) \cup (\tau_\bv \setminus \{z\})$ where $z \in \zeta$. Therefore
$$ \lk_{\Sigma^{[d]}}(\sigma_2)=
\langle \overline{\tau_\bv}\setminus (\tau_\bm \cup \sigma_2') \rangle * \partial \langle\zeta \rangle
% \begin{cases}
%  \partial \langle\zeta \rangle  & \qif \sigma_2'' \neq \emptyset\\
%  \langle \tau_\bv \rangle & \qif \sigma_2'' = \emptyset. 
% \end{cases}
$$ 
Hence $\lk_{\Sigma^{[d]}}(\sigma_2)$ has either no homology, or only one nontrivial homology group if 
$\overline{\tau_\bv}\setminus (\tau_\bm \cup \sigma_2')=\emptyset$; in other words 
 $$\rH_b(\lk_{\Sigma^{[d]}}(\sigma_2))=0 \qforall b\neq \dim (\lk_{\Sigma^{[d]}}(\sigma_2))=
 d-|\sigma_2|.$$
 In view of \eqref{e:H*1} and \eqref{e:j-deg}, we can now conclude that 
\begin{equation}\label{e:SO-cap}
\rH_j(\lk_{\Omega^{[i]}}(\sigma) 
\cap 
(\langle \tau_\bm\setminus \sigma_1 \rangle * \lk_{\Sigma^{[d]}}(\sigma_2)))=0 \qif j<i -|\sigma|.
\end{equation}
\end{enumerate}

In  conclusion: \eqref{e:SO-cap} along with \eqref{i:sum} and \eqref{i:cup} applied to the Mayer-Vietoris sequence in \eqref{e:MV-new} allows us to conclude 
$$\rH_j(\lk_{\Omega^{[i]}}(\sigma))\cong \rH_j(\lk_{\Gamma^{[i]}}(\sigma))=0 \qif 
j<i-|\sigma|=\dim(\lk_{\Omega^{[i]}}(\sigma)).$$ 

By Reisner's criterion for Cohen-Macaulayness~(\cite{Reisner}), we can conclude that $\Omega^{[i]}$ is a Cohen-Macaulay complex for every $i$, and hence $I+(\bm)$ is a sequentially Cohen-Macaulay ideal.
\end{proof}

Applying \cref{t:polar} to \cref{t:scm} yields the following more general result.

\begin{corollary}[{\bf Expansion sequential preserves Cohen-Macaulay properties}]\label{c:scm} Let $I$ be a  monomial ideal in $S=\kappa[x_1,\ldots,x_n]$ and let $\bn$ and $\bm$ be  monomials such that $\bn \in \mingens (I)$ and  $\m\in C_{I}(\bn)$. If $I$ is sequentially Cohen-Macaulay, then so is $I+(\bm)$.
\end{corollary}

%%%%%%%%
\bibliographystyle{plain}
\bibliography{ref}
\end{document}